\documentclass[11pt]{amsart}
\usepackage{epsfig}
\usepackage{graphics}

\newtheorem{theorem}{Theorem}

\newtheorem{proposition}[theorem]{Proposition}

\theoremstyle{definition}

\newtheorem{remark}[theorem]{Remark}

\theoremstyle{remark}

 \def\Z{{\mathbb{Z}}}
 \def\C{{\mathbb{C}}}

  \def\D{{\mathbb{D}}}

\begin{document}

\newenvironment{prooff}{\medskip \par \noindent {\it Proof}\ }{\hfill
$\square$ \medskip \par}
    \def\sqr#1#2{{\vcenter{\hrule height.#2pt
        \hbox{\vrule width.#2pt height#1pt \kern#1pt
            \vrule width.#2pt}\hrule height.#2pt}}}
    \def\square{\mathchoice\sqr67\sqr67\sqr{2.1}6\sqr{1.5}6}
\def\pf#1{\medskip \par \noindent {\it #1.}\ }
\def\endpf{\hfill $\square$ \medskip \par}
\def\demo#1{\medskip \par \noindent {\it #1.}\ }
\def\enddemo{\medskip \par}
\def\qed{~\hfill$\square$}

 \title[]
 {Open book decompositions of links of quotient surface singularities and support genus problem}
\begin{center}
\end{center}

 \author{Elif Dalyan}

 \address{Department of Mathematics, Hitit University,
 Corum, Turkey} \email{elifdalyan@hitit.edu.tr, elif@elifyilmaz.com.tr}

 \date{\today}
 \keywords{Quotient surface singularities, open book decomposition, contact structure, Milnor genus, support genus}

\begin{abstract}
In this paper we write explicitly the open book decompositions of links of quotient surface singularities 
supporting the corresponding unique Milnor fillable contact structure. The page-genus of these Milnor open books 
are minimal among all Milnor open books supporting the same contact structure. We also investigate whether the Milnor genus is equal to the support genus for links of quotient surface singularities. We show that for many types of the quotient surface singularities the Milnor genus is
equal to the support genus. In the remaining cases we are able to find a small upper bound for the support genus.
\end{abstract}

 \maketitle
  \setcounter{secnumdepth}{2}
 \setcounter{section}{0}

\section{Introduction}

The purpose of this paper is to construct the Milnor open book decompositions of the links of quotient surface singularities supporting the unique Milnor fillable contact structure. By the work of Bhupal--Alt{\i}nok \cite{ab} and by Nemethi--Tosun \cite{nt},
the page-genus of our Milnor open book is minimal among all Milnor open books supporting the same contact structure, i.e. it gives the Milnor genus. In \cite{bo2}, it is shown that for some examples of rational surface singularities Milnor genus is not equal to the support genus. However, if we restrict ourself to quotient surface singularities, the question whether the Milnor genus
is equal to the support genus for the canonical contact structure
is still unknown. For most cases of the quotient surface singularities, we provide planar Milnor open books,
so that for these
types the Milnor genus is equal to the support genus. In all remaining cases, the Milnor genus turns out
 to be one. Hence, the
support genus of the corresponding contact structure is at most one. We are able to show that for
some of these quotient surface singularities the Milnor genus is equal to the support genus, which is one.

Our main result is the following theorem.

\begin{theorem}\label{mainth}
The unique Milnor fillable contact structure on the link of the quotient surface singularities has support
genus one for each singularities of the following types:
\begin{itemize}
	\item Tetrahedral part $(i)$ where $b = 2$ (cf. Figure~\ref{tetrahedral}).
	\item Octahedral part $(i)$ where $b = 2$(cf. Figure~\ref{octahedral}).
	\item Icosahedral part $(i)$ and $(ii)$ where $b = 2$. (cf. Figure~\ref{icosahedral1}).
\end{itemize}	
The support genus is zero for singularities of the following types:
\begin{itemize}
	\item Cyclic (cf. Figure~\ref{cyclic}).
	\item Dihedral, $b>2$ (cf. Figure~\ref{dihedral}).
	\item Tetrahedral, $b>2$ (cf. Figure~\ref{tetrahedral}).
	\item Octahedral, $b>2$ (cf. Figure~\ref{octahedral}).
	\item Icosahedral, $b>2$ (cf. Figure~\ref{icosahedral1}).
\end{itemize}	
For the remaining cases, the corresponding contact structures have support genus at most one.
\end{theorem}

For the cyclic singularity, the planar open book decompositions were constructed in \cite{s}.
For icosahedral singularity of part $(i)$ where $b=2$, it was shown in \cite{e} that contact
structure cannot be supported by a planar open book decomposition. Moreover,  it was shown in \cite{b} and \cite{eo} that
this singularity has a genus--one open book decomposition supporting that contact structure.
\vspace{0.5cm}

\textbf{Acknowledgements:} This paper is a part of my Ph.D. thesis \cite{y} at Middle East Technical University.
Special thanks to Kaoru Ono for proposing this question, to Andr\'{a}s Stipsicz, Mohan Bhupal and Mustafa Korkmaz for their helpful comments.

\section{Preliminaries} \label{sec:qssing}

\subsection{Quotient Surface Singularities}
We study the quotient singularities $\C^2 / G $, where $G$ is a finite subgroup of $GL(2,\C)$. Brieskorn~\cite{br} described the possible minimal resolutions for these singularities, by using earlier result of Prill~\cite{p}. These singularities are classified into five groups, namely cyclic quotient singularities, dihedral singularities, tetrahedral singularities, octahedral singularities and icosahedral singularities.
We give the minimal resolution graphs of these singularities, and use it in our proof. The reader is referred to \cite {bo} for more details.

\begin{enumerate}
	\item[$\bullet$] \textbf{Cyclic Quotient Singularities:} $A_{n,q}$, where $0<q<n$ and $\gcd(n,q)=1$. The minimal resolution graph
of $A_{n,q}$ is given in Figure~\ref{cyclic}, where $b_{i}$ are defined by the continued fraction
 \[
 \frac{n}{q}= \left[b_{1},b_{2},\ldots ,b_{r} \right]= b_{1}- \frac{1}{b_{2}-\frac{1}{\ddots-\frac{1}{b_{r}}}}
 \]
 with $b_{i} \geq 2$ for all $i$.
\begin{figure}[h]

    \includegraphics[width=9cm]{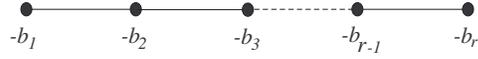}
 \caption{Cyclic quotient singularity.}
  \label{cyclic}

 \end{figure}

\item[$\bullet$]
\textbf{Dihedral Quotient Singularities:} The minimal resolution graph of a dihedral quotient singularity is
given in Figure~\ref{dihedral}, where $b\geq 2$ and $ b_{i}\geq 2$.
 \begin{figure}[hbt]
 \begin{center}
    \includegraphics[width=6cm]{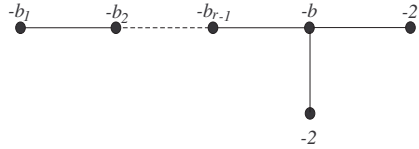}
  \caption{Dihedral quotient singularities.}
  \label{dihedral}
   \end{center}
 \end{figure}

\item[$\bullet$]
\textbf{Tetrahedral Singularities:} The minimal resolution  graph of a tetrahedral singularity is
given in Figure~\ref{tetrahedral}, where $b\geq 2$.
 \begin{figure}[hbt]
 \begin{center}
    \includegraphics[width=8cm]{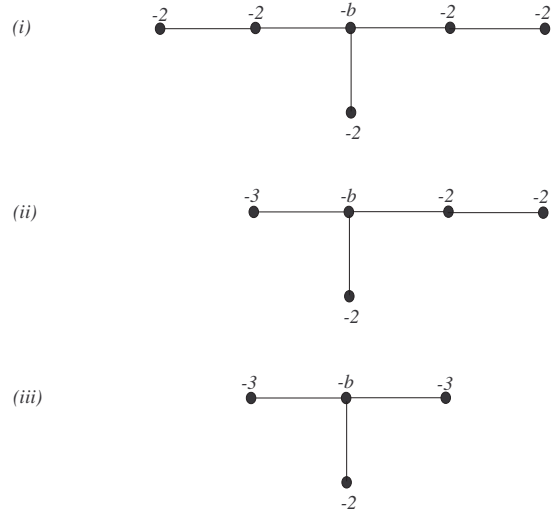}
   \caption{Tetrahedral singularities.}
  \label{tetrahedral}
   \end{center}
 \end{figure}

\item[$\bullet$]
\textbf{Octahedral Quotient Singularities:} The minimal resolution  graph of a octahedral quotient singularity is
of the form given in Figure~\ref{octahedral}, where $b\geq 2$.
 \begin{figure}[hbt]
 \begin{center}
    \includegraphics[width=9cm]{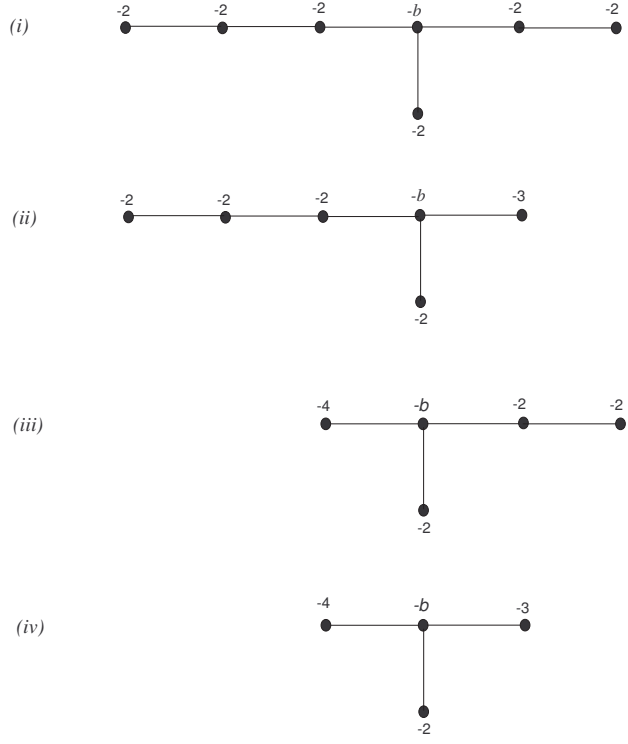}
  \caption{Octahedral quotient singularities.}
  \label{octahedral}
   \end{center}
 \end{figure}

\item[$\bullet$]
\textbf{Icosahedral Quotient Singularities:}
The minimal resolution graph of a icosahedral quotient singularity is
of the form given in Figure~\ref{icosahedral1}, where $b\geq 2$.
  \begin{figure}[hbt]
 \begin{center}
    \includegraphics[width=8cm]{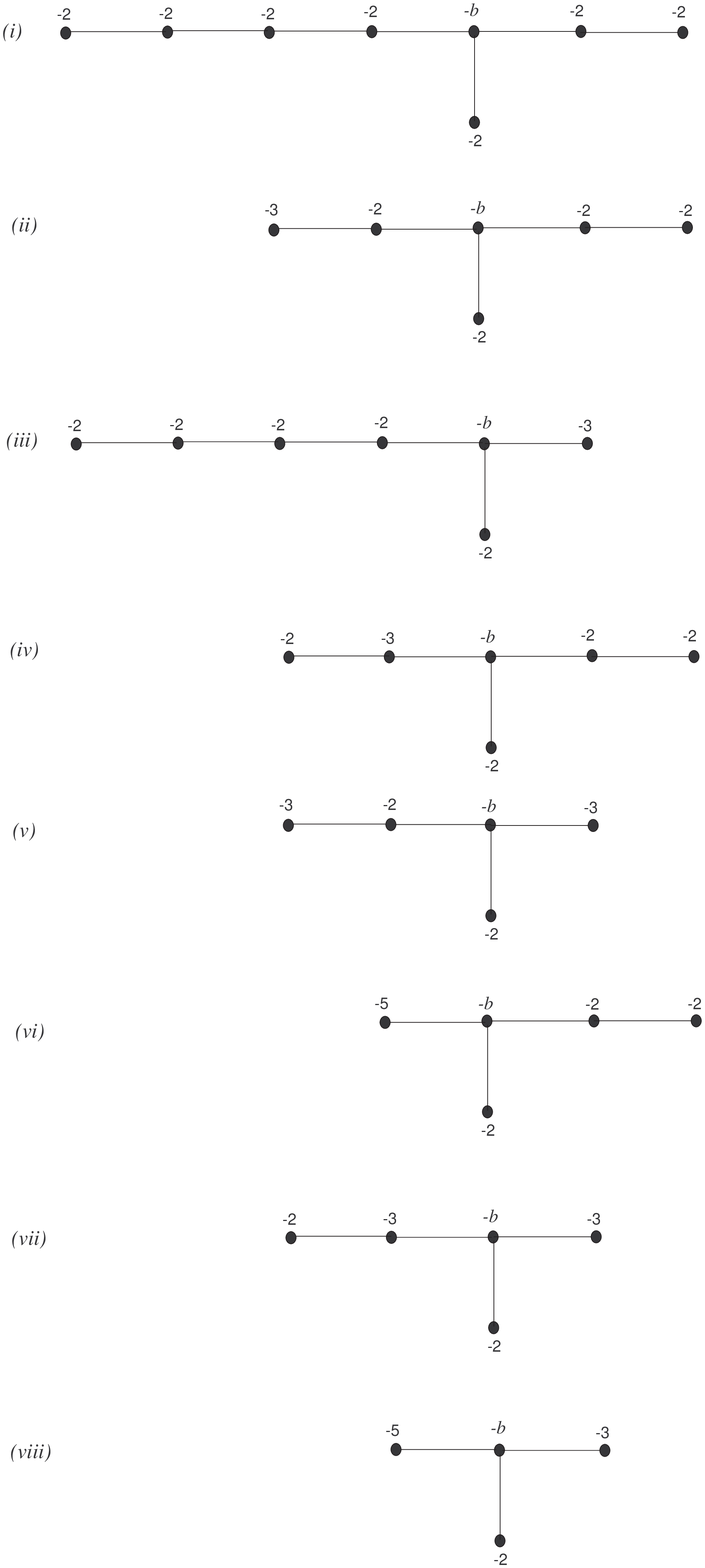}
  \caption{Icosahedral quotient singularities.}
  \label{icosahedral1}
   \end{center}
 \end{figure}
\end{enumerate}

\subsection{Mapping class groups}
The mapping class group $MCG(\Sigma)$ of a compact connected orientable surface $\Sigma$
is defined as the group of isotopy classes of orientation--preserving
self--diffeomorphisms of $\Sigma$, where diffeomorphisms
and isotopies of $\Sigma$ are assumed to be the identity on the boundary. The group $MCG(\Sigma)$
is generated by Dehn twists.

We need the following torus relations. These relations can be obtained from the well known one--holed torus relation by using the lantern and braid relations. The reader is referred to \cite{ko} for the details. For the curves in the relations see the appropriate picture in Figure~\ref{relations}.
\begin{figure}[ht]
    \includegraphics[width=12cm]{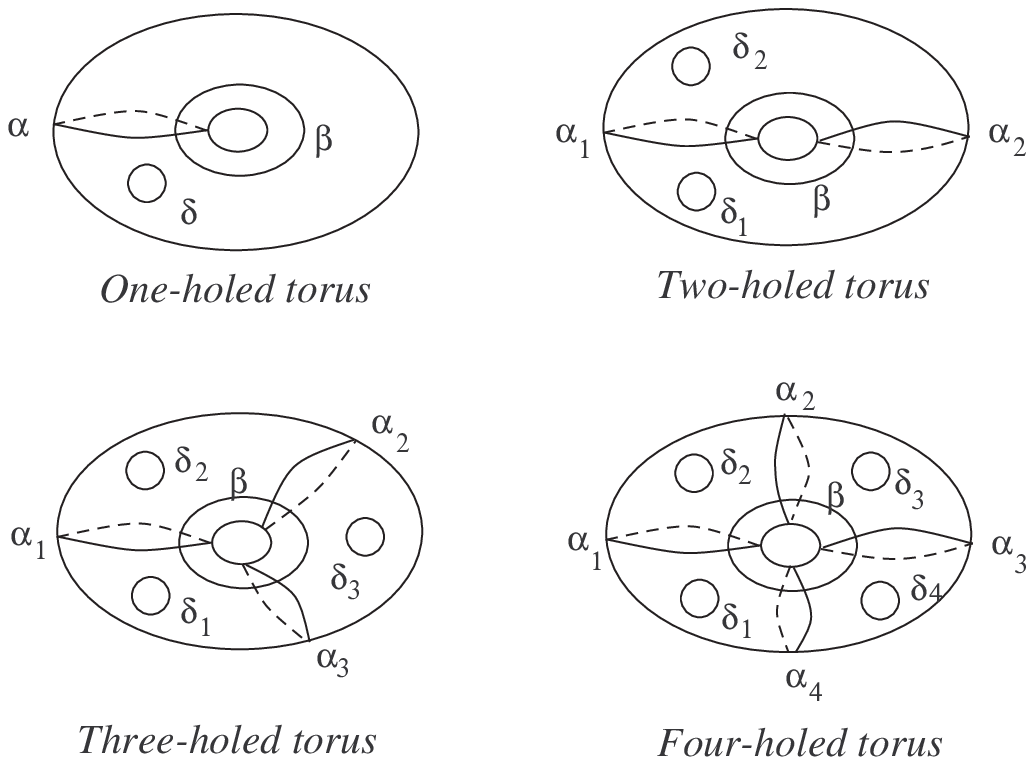}
 \caption{Curves of torus relations.}
  \label{relations}
 \end{figure}

One--holed torus relation is
\begin{equation} \label{1holed}
t_{\delta}  = (t_{\alpha} t_{\beta})^{6} .
\end{equation}

Two--holed torus relation is
\begin{equation} \label{2holed1}
t_{\delta_{1}} t_{\delta_{2}}  = (t_{\alpha_{1}} t_{\alpha_{2}} t_{\beta})^4,
\end{equation}
or by using braid relations, equivalently we can write
\begin{equation} \label{2holed2}
t_{\delta_{1}} t_{\delta_{2}}  = (t_{\alpha_{1}} t_{\alpha_{2}}t_{\alpha_{2}} t_{\beta})^{3},
\end{equation}
or
\begin{equation} \label{2holed3}
t_{\delta_{1}} t_{\delta_{2}}  =  (t_{\alpha_{1}} t_{\alpha_{2}} t_{\beta} t_{\alpha_{2}}t_{\alpha_{2}} t_{\beta})^{2}.
\end{equation}

Three--holed torus relation is
\begin{equation} \label{3holed1}
t_{\delta_{1}} t_{\delta_{2}} t_{\delta_{3}} = (t_{\alpha_{1}} t_{\alpha_{2}} t_{\alpha_{3}} t_{\beta})^{3} .
\end{equation}
or
\begin{equation} \label{3holed2}
t_{\delta_{1}} t_{\delta_{2}} t_{\delta_{3}} = (t_{\alpha_{1}} t_{\alpha_{3}} t_{\beta} t_{\alpha_{2}} t_{\alpha_{3}} t_{\beta})^{2},
\end{equation}

Four--holed torus relation is
\begin{equation} \label{4holed}
t_{\delta_{1}} t_{\delta_{2}} t_{\delta_{3}} t_{\delta_{4}} = (t_{\alpha_{1}} t_{\alpha_{3}} t_{\beta} t_{\alpha_{2}} t_{\alpha_{4}} t_{\beta})^{2} .
\end{equation}

The next theorem was proved by C. Bonatti and L. Paris (c.f \cite{bp}, Theorem 3.6). It will be useful for us in writing the roots of elements in the mapping class group of a torus with boundary.

\begin{theorem}\label{mcg}
If $\Sigma$ is a torus with non-empty boundary components, then each element $f$ in $MCG(\Sigma)$ has at most one $m$-root up to conjugation for all $m\geq 1$.

\end{theorem}

\section{Construction of Milnor open books} \label{costruction}

We write the minimal page-genus Milnor open book decompositions of the links of the quotient surface singularities with the help of the minimal resolution graphs given in Section~\ref{sec:qssing}. We will give a recipe for the construction of these open books. The reader is referred to~\cite{b} for a detailed explanation.

 Let $\Gamma$ be one of the graphs given in Section~\ref{sec:qssing}. The intersection matrix $I(\Gamma)$
 of $\Gamma$ is the negative definite symmetric matrix defined as follows: First label the vertices of
 $\Gamma$ as $A_1,A_2,\ldots,A_q$.  \textit{We index the vertices of the graph starting from left to right and then index the bottom vertex if exists.}
 The $(i,i)$ entry of $I(\Gamma)$ is the weight associated
 to $A_i$. For $i\neq j$, the $(i,j)$ entry is defined as $1$ (resp. $0$) if $A_i$ is connected
 (resp. not connected) to $A_j$.

 In order to construct our open book, we first find $1\times q$ integer matrices
$\underline{m}=\left[
                 \begin{array}{cccc}
                   m_1  & m_2 &\cdots & m_q \\
                 \end{array}
               \right]$
and
$\underline{n}=\left[
                 \begin{array}{cccc}
                   n_1  & n_2 &\cdots & n_q \\
                 \end{array}
               \right]$
satisfying
\begin{equation}  \label{first}
I(\Gamma)\, \underline{m}^{t}=-\underline{n}^{t}.
\end{equation}
We choose $\underline{m}$ in such a way that $m_i$ are the smallest
possible positive integers so that $n_{i}\geq 0$ for all $i$.
Here, $\underline{m}^t$ denotes the transpose
of the matrix $\underline{m}$.

The page $\Sigma$ of the open book associated to $\underline{m}$ and $\underline{n}$
satisfying the equality~\eqref{first} is a union of the following pieces: A collection of
surfaces $F_{i}$, for $i=1,\ldots,q$; annuli $U^{i}_{t}$, for $i=1,\ldots,q$, $t=1,\ldots,n_{i}$;
and a collection of annuli $U^{i,j}_{l}$, $l=1,\ldots,\gcd(m_{i},m_{j})$ for each pair $(i,j)$
with $1\leq i < j\leq q$ such that $(A_{i},A_{j})\in \mathcal{E} $, where $\mathcal{E}$ denotes
the set of edges of the graph $\Gamma $.

We determine the surface $F_i$ as follows:
For the vertex $A_{i}$ with valency $v_{i}$, $F_{i}$ is an $m_{i}$--cover of the sphere with $v_{i}+ n_{i}$ boundary components. The genus $g(F_i)$ of $F_i$ is determined by the followings:
If $n_{i} > 0$ then the surface $F_{i}$ is connected and
\begin{equation}  \label{second}
2-2g(F_{i})- \sum_{(A_{i},A_{j})\in \mathcal{E}} \gcd{(m_{i},m_{j})}-n_{i}=m_{i}(2-v_{i}-n_{i}),
\end{equation}
from which we obtain
\begin{equation}  \label{third}
g(F_{i})= 1+ \frac{m_{i}(v_{i}+n_{i}-2)- \sum_{(A_{i},A_{j})\in \mathcal{E}} \gcd{(m_{i},m_{j})}-n_{i}}{2}.
\end{equation}
If $n_{i}=0$ then $F_{i}$ has $d_{i}= \gcd{(\left\{m_{i}\right\}\cup \left\{m_{j}|(A_{i},A_{j})\in \mathcal{E}\right\})}$ connected components and the genus of these components $F^{s}_{i}$, $s=1,\ldots, d_{i}$
is calculated as
\begin{equation}  \label{fourth}
g(F^{s}_{i})= 1+ \frac{ ( m_{i} / d_{i}) (v_{i}-2)- \sum_{(A_{i},A_{j})\in \mathcal{E}} \gcd{(m_{i},m_{j})} / d_{i}}{2}.
\end{equation}

The number of boundary components of $F_{i}$ is
\[
n_{i}+ \sum_{(A_{i},A_{j})\in \mathcal{E}} \gcd{(m_{i},m_{j})}.\]
 After gluing those surfaces according to the graph $\Gamma$, we end up with the page $\Sigma$ of the open book decomposition.

For a simple closed curve $a$ on an oriented surface, let us denote by $t_a$ the right Dehn twist about $a$. In order to find the monodromy $\phi$ of the open book decomposition, we first find the monodromy restricted to the annuli $U^{i}_{t}$ and $U^{i,j}_{l}$'s which make up the page $\Sigma$ together with the surfaces $F_{i}$'s.  The monodromy $\phi$ restricted to the annulus $U^{i}_t$ is given by
\[
\left(\phi|_{U^{i}_{t}}  \right)^{m_{i}}=t_{\delta^{i}_{t}},
\]
for $i=1,\ldots, N$, where $\delta^{i}_{t}$ is the core circle of $U^{i}_{t}$ and, hence, it is parallel to the boundary components of the page $\Sigma$.
The monodromy restricted to annulus $U^{i,j}_{l}$ is given by
\[
\left( \phi|_{U^{i,j}_{l}} \right)  ^{m_{i}m_{j}/\gcd{(m_{i},m_{j})}}= t_{c_{j-1}},
\]
where $c_{j-1}$ is the core of the annulus  $U^{i,j}_{l}$. We glue these diffeomorphisms to get the monodromy $\phi:\Sigma \to \Sigma$ of the open book decomposition.

Let us now construct the open book explicitly for all types of singularities.

\section{Milnor open book decompositions} \label{lemmas}
In this section we give the whole list of Milnor open book decompositions supporting
the corresponding unique Milnor fillable contact structure on the links of quotient surface singularities.

\subsection{Cyclic quotient singularities}
We start by investigating the cyclic quotient singularities. We construct the Milnor open book supporting the unique Milnor fillable contact structure by following the construction steps explained in the previous section. This is the easiest part and the open books turned to be planar.
\begin{proposition} \label{prop:cyclic} {\rm \textbf{(Cyclic quotient singularities)}}
The unique Milnor fillable contact structure on the link of a cyclic quotient surface
singularity is supported by a planar open book decomposition with $N=b_{1}+b_{2}+ \cdots +b_{r}- 2(r-1)$
boundary components. The monodromy of the open book is
\[
\phi=(t_{\delta^{1}_{1}}\cdots t_{\delta^{1}_{b_{1}-1}}) \cdots ( t_{\delta^{i}_{1}}\cdots t_{\delta^{i}_{b_{i}-2}} ) \cdots  ( t_{\delta^{r}_{1}}\cdots t_{\delta^{r}_{b_{r}-1}} ) ( t_{c_{1}}\cdots t_{c_{r-1}} ),
\]
where $i=2,\ldots,r-1$.

\end{proposition}

\begin{proof}
The intersection matrix $I(\Gamma)$ of the cyclic quotient surface singularity is
 \[
\begin{bmatrix}
  -b_{1} & 1     &   0     &   \cdots &   0&   0 & 0\\
   1     & -b_{2}&   1     &   \cdots  &   0&   0   & 0  \\
   0     & 1     & -b_{3}&     \cdots &   0 &   0   & 0  \\
     \vdots    & \vdots&  \vdots &  \vdots& \vdots &   \vdots & \vdots  \\
 0 & 0 &   0  &   \cdots  & 1       & -b_{r-1} &  1 \\
  0      & 0&   0 &  \cdots   &   0   &   1      &  -b_{r}\\
\end{bmatrix}.
\]
Consider the $r$-tuple of integers $\underline{m}=(1,1,\ldots,1)$, which gives the fundamental cycle of the resolution (We see $\underline{m}$ as a matrix). Then we find that $n_1=b_1-1,n_r=b_r-1$ and $n_i=b_i-2$
for $i=2,3,\ldots, r-1$.
The page $\Sigma$ of the open book associated to $\underline{m}$ consists of the following pieces. A collection of surfaces $F_{i}$, annuli $U^{i}_{t}$ for each binding component of the open book and a collection of annuli $U^{i,j}_{l}$ connecting the surfaces $F_{i}$ and  $F_{j}$. Notice that the equations (\ref{third}) and (\ref{fourth}) become the same in both cases $n_{i} > 0$ and $n_{i}=0$. By using either of these equations, we find that $g(F_{i})=0$ for all $i$. The number of boundary components of $F_{i}$ is $n_{i} + \sum_{(A_{i},A_{j})\in \mathcal{E}} \gcd{(m_{i},m_{j})}$. It follows that each $F_{i}$ is a sphere with $b_{i}$ boundary components.

Next, we glue the annuli to the $F_{i}$'s. For each $1\leq i\leq r-1$, glue the annulus $U^{i,i+1}_{1}$ to
$F_i$ and $F_{i+1}$ to get a connected surface (cf. Figure~\ref{pagecyclic}).
There are $n_{i}$ annuli $U^{i}_{t}$ which are not used to plumb the surfaces $F_i$. They will give the
binding components of the open book. As seen in Figure \ref{pagecyclic}, the page $\Sigma$ is a sphere with $N$ boundary components, where
\begin{eqnarray*}
  N &=& n_1+n_2+\cdots +n_r \\
   &=& b_{1}+b_{2}+ \cdots +b_{r}- 2(r-1).
\end{eqnarray*}

In order to find the monodromy $\phi$, we only need to find $\phi|_{U^{i}_{t}}$ and $\phi|_{U^{i,j}_{l}}$. We know that $(\phi|_{U^{i}_{t}})^{m_{i}}=t_{\delta^{i}_{t}}$, for $i=1,\ldots, N$, where $\delta^{i}_{t}$ are the core circles of $U^{i}_{t}$. We find that the monodromy restricted to each annulus $U^{i,j}_{l}$ is given by  $(\phi|_{U^{i,j}_{l}})^{m_{i}m_{j}/\gcd{(m_{i},m_{j})}}= t_{c_{j-1}}$, where $c_{j-1}$ is the core of the annulus  $U^{i,j}_{l}$ (cf. Figure \ref{pagecyclic}). Since $m_i=1$, we have
\begin{enumerate}
  \item[$\bullet$] $\phi|_{U^{1}_{j}}= t_{\delta^{1}_{j}}$ for $j=1,\ldots,b_{1}-1$,
  \item[$\bullet$] $\phi|_{U^{i}_{j}}= t_{\delta^{i}_{j}}$ for $i=2, \ldots , r-1$ and $j=1, \ldots ,b_{i}-2$,
  \item[$\bullet$] $\phi|_{U^{r}_{j}}= t_{\delta^{r}_{j}}$ for $j=1, \ldots , b_{r}-1$, and
  \item[$\bullet$] $\phi|_{U^{i,i+1}_{1}}= t_{c_{i}}$ for $i=1,\ldots,r-1$.
\end{enumerate}
By gluing these maps, we find that the monodromy is
\[
\phi=(t_{\delta^{1}_{1}}\cdots t_{\delta^{1}_{b_{1}-1}}) \cdots ( t_{\delta^{i}_{1}}\cdots t_{\delta^{i}_{b_{i}-2}} ) \cdots  ( t_{\delta^{r}_{1}}\cdots t_{\delta^{r}_{b_{r}-1}} ) ( t_{c_{1}}\cdots t_{c_{r-1}} ),
\]
where $i=2,\ldots,r-1$.

\end{proof}

  \begin{figure}[hbt]
 \begin{center}
    \includegraphics[width=9cm]{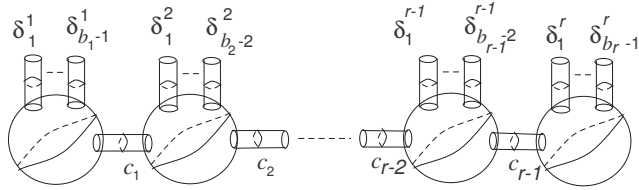}
  \caption{The page for cyclic singularities.}
  \label{pagecyclic}
   \end{center}
 \end{figure}

\subsection{Dihedral quotient singularities}
First, we consider the dihedral quotient surface singularity.
If $b_1 =b_2= \cdots = b_{r-1}=b=2$, then the singularity is a simple surface
singularity. In~\cite{b}, Bhupal computes the open book decomposition:
The page of the open book is a one--holed torus and the monodromy is
$\left( t_{\alpha}t_{\beta} \right)^3 \left( t_{\alpha} \right)^{r-2}$. In the following proposition, we assume that the surface singularity is not simple.

\begin{figure}[hbt]
 \begin{center}
    \includegraphics[width=12cm]{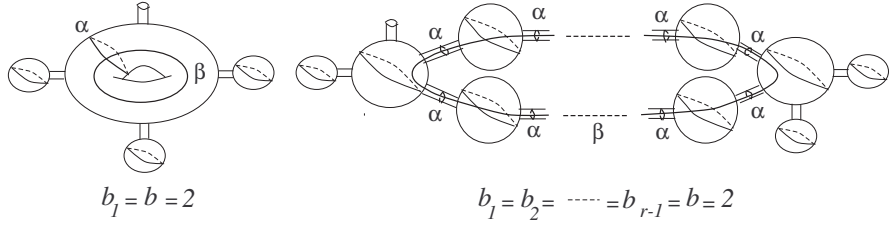}
  \caption{The page for the dihedral simple surface singularity for $r=2$ and $r>2$.}
  \label{dihedralsimple}
   \end{center}
 \end{figure}

\begin{proposition} \label{prop:dihedral}{\rm \textbf{(Dihedral quotient singularities)}}
If $b > 2$ then the unique Milnor fillable contact structure on the link of the dihedral quotient surface
singularity is supported by a planar open book decomposition with $N=b_{1}+b_{2}+ \cdots + b_{r-1}+b - 2r+1$
boundary components. The monodromy of the open book is
\[
\phi= T_1 T_2 \cdots T_{r-1} T_r  \left( t_{\delta^{r+1}_{1}} \right)^{2}
\left( t_{\delta^{r+2}_{1}} \right)^{2} \left( t_{c_{1}}\cdots t_{c_{r-1}} \right).
\]

If $b =b_{r-1}=b_{r-2}= \cdots = b_{k+1} =2$ and $b_k>2$ for some $1\leq k\leq r-1$,
then the unique Milnor fillable contact structure on the link of the dihedral quotient surface
singularity is supported by an open book of genus one. The number of boundary components  is $N=b_{1}+b_{2}+ \cdots +b_{k}- 2k+1$. The monodromy is given by
\[\displaystyle
\phi =\left\{
\begin{array}{ll}\displaystyle
    \left( t_{\delta^{1}_{1}} t_{\delta^{1}_{2}} \cdots t_{\delta^{1}_{b_{1}-2}}\right) \left(t_{\alpha_{1}}t_{\alpha_{2}} t_{\beta}  \right)^{2}, \mbox{ if }r=2 \ (\mbox{hence } k=1) \\
    T_1 T_2 \cdots T_{k-1} W_{k} ( t_{c_{1}}\cdots t_{c_{k-1}} ) (t_{\alpha_{1}}t_{\alpha_{2}} t_{\beta} )^{2} \left( t_{\alpha_{2}}\right) ^{r-(k+1)}, \mbox{ if }r>2.
\end{array}
\right.
\]
Here,
\begin{itemize}
  \item[] $T_1=t_{\delta^{1}_{1}}\cdots t_{\delta^{1}_{b_{1}-1}}$,
  \item[] $T_i=  t_{\delta^{i}_{1}}\cdots t_{\delta^{i}_{b_{i}-2}}$, $i=2,\ldots,r-1$.
  \item[] $T_r=t_{\delta^{r}_{1}}\cdots t_{\delta^{r}_{b-3}}$, \ and
  \item[] $W_k=t_{\delta^{k}_{1}}\cdots t_{\delta^{k}_{b_{k}-3}}$.
\end{itemize}

\end{proposition}

\begin{proof}
 We construct the open book by following the steps explained in Section~\ref{costruction}.

Suppose first that $b > 2$. Let $\underline{m} = (1, \ldots , 1)$.
From the equation~\eqref{first}, we
get $\underline{n}= (b_{1}-1,b_{2}-2, \ldots , b_{r-1}-2, b - 3,1,1)$.
For each vertex $A_i$ of the graph with valency $v_i$, we take a sphere $S_i$
with $v_i+n_i$ boundary components. Since $m_i=1$, $S_i=F_i$ in the notation of Section~\ref{costruction}. For each edge $E$ of the graph connecting the vertices $A_i$ and $A_j$, we glue an annulus connecting the spheres $S_i$ and $S_j$. We then glue $n_i$ annuli to the sphere $S_i$. The resulting surface is a page of the open book and is a sphere with
$N=n_1+\cdots +n_{r+2}=b_{1}+b_{2}+ \cdots + b_{r-1}+b - 2r+1$ boundary components (cf. Figure~\ref{pagedihedalcyclic}). For the monodromy, we have the diffeomorphisms below, gluing them we find the monodromy of the open book.

In Section~\ref{costruction}, it is explained that $(\phi|_{U^{i}_{t}})^{m_{i}}=t_{\delta^{i}_{t}}$, for $i=1,\ldots, N$, where $\delta^{i}_{t}$ are the core circles of $U^{i}_{t}$. We find that the monodromy restricted to each annulus $U^{i,j}_{l}$ is given by  $(\phi|_{U^{i,j}_{l}})^{m_{i}m_{j}/\gcd{(m_{i},m_{j})}}= t_{c_{j-1}}$, where $c_{j-1}$ is the core of the annulus  $U^{i,j}_{l}$. Then we have
\begin{enumerate}
  \item[$\bullet$] $\phi|_{U^{1}_{j}}= t_{\delta^{1}_{j}}$ for $j=1,\ldots,b_{1}-1$,
  \item[$\bullet$] $\phi|_{U^{i}_{j}}= t_{\delta^{i}_{j}}$ for $i=2, \ldots , r-1$ and $j=1, \ldots ,b_{i}-2$,
  \item[$\bullet$] $\phi|_{U^{r}_{j}}= t_{\delta^{r}_{j}}$ for $j=1, \ldots , b-3$,
  \item[$\bullet$] $\phi|_{U^{r+1}_{1}}= t_{\delta^{r+1}_{1}}$,
  \item[$\bullet$] $\phi|_{U^{r+2}_{1}}= t_{\delta^{r+2}_{1}}$, and
  \item[$\bullet$] $\phi|_{U^{i,i+1}_{1}}= t_{c_{i}}$ for $i=1,\ldots,r+1$.
\end{enumerate}

The curve $c_r$ is isotopic to the curve $\delta^{r+1}_{1}$ and the curve $c_{r+1}$ is isotiopic to the curve $\delta^{r+2}_{1}$, and gluing the maps above we can easily get the monodromy.

\[
\phi=(t_{\delta^{1}_{1}}\cdots t_{\delta^{1}_{b_{1}-1}}) \cdots ( t_{\delta^{i}_{1}}\cdots t_{\delta^{i}_{b_{i}-2}} ) \cdots  ( t_{\delta^{r}_{1}}\cdots t_{\delta^{r}_{b-3}} )\left( t_{\delta^{r+1}_{1}} \right)^{2}
\left( t_{\delta^{r+2}_{1}} \right)^{2} \left( t_{c_{1}}\cdots t_{c_{r-1}} \right),
\]
where $i=2,\ldots,r-1$.

\[
\phi= T_1 T_2 \cdots T_{r-1} T_r  \left( t_{\delta^{r+1}_{1}} \right)^{2}
\left( t_{\delta^{r+2}_{1}} \right)^{2} \left( t_{c_{1}}\cdots t_{c_{r-1}} \right).
\]

  \begin{figure}[hbt]
 \begin{center}
    \includegraphics[width=9cm]{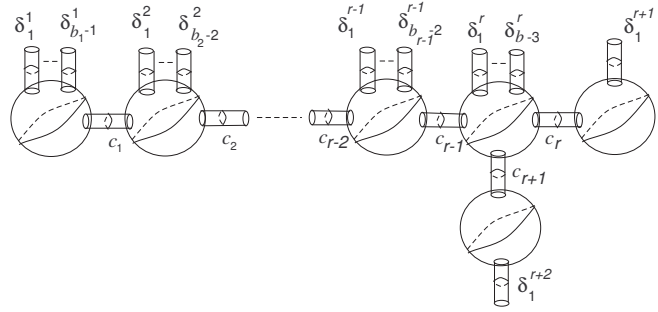}
  \caption{The page for dihedral singularities for $b>2$.}
  \label{pagedihedalcyclic}
   \end{center}
 \end{figure}

Suppose now that $b=2$. We consider two cases:

{\bf Case 1: $r = 2$.}
In this case by taking $ \underline{m} =(1,2,1,1)$, we find from equation~(\ref{first}) that
$\underline{n}= (b_1-2,1,0,0)$. It follows from the construction that
the page of the open book is a torus with $b_1-1$ boundary components (cf. Figure \ref{dihedral0}). For the monodromy, there are the diffeomorphisms of the annuli $U^{2}_{1}$, $U^{1}_{i}$ and $U^{1,2}_{1}$ given by
\begin{enumerate}
  \item[] $\left(  \phi|_{U^{2}_{1}} \right)^2 = t_{\delta^{2}_{1}}$ ,
  \item[] $\left(  \phi|_{U^{1,2}_{1}} \right)^2 = t_{c_{1}}$, and
  \item[] $\left(  \phi|_{U^{1}_{i}} \right)= t_{\delta^{1}_{i}}$ , for $i=1,2, \ldots, b_1-2$.
\end{enumerate}
By using the two-holed torus relation (\ref{2holed1}) for the torus bounded by $c_{1}$ and $\delta ^{2}_{1}$, and by Theorem~\ref{mcg}, the monodromy is found to be
\begin{eqnarray*}
  \phi
   &=& \left( (\phi|_{U^{1}_{1}} ) (\phi|_{U^{1}_{2}})  \cdots  (\phi|_{U^{1}_{b_{1}-2}}) \right) \left(  (\phi|_{U^{2}_{1}}) ( \phi|_{U^{1,2}_{1}}) \right) \\
   &=& \left( t_{\delta^{1}_{1}} t_{\delta^{1}_{2}} \cdots t_{\delta^{1}_{b_{1}-2}}\right) \left(t_{\alpha_{1}}t_{\alpha_{2}} t_{\beta}  \right)^{2}.
\end{eqnarray*}

\begin{figure}[hbt]
 \begin{center}
    \includegraphics[width=6cm]{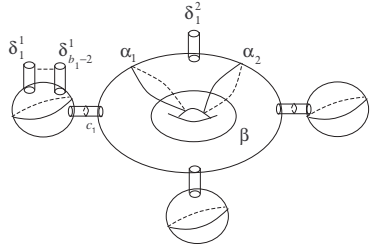}
  \caption{The page $\Sigma$ for dihedral singularity for $r=2$, $b = 2$ and $b_1 > 2$.}
  \label{dihedral0}
   \end{center}
 \end{figure}

 {\bf Case 2: $r > 2$.} We divide this case into two subcases: $k=r-1$ and $k<r-1$.

Suppose first that $k=r-1$, so that $b_{r-1} > 2$.  In this case we take
\[\underline{m} =(1,1,\ldots,1,1,2,1,1)\]
so that all $m_i=1$ but $m_r=2$. We easily find from equation~(\ref{first}) that
\[
\underline{n}= (b_{1}-1,b_{2}-2,b_3-2, \ldots , b_{r-2}-2, b_{r-1}-3,1,0,0).
\]
It follows that the page of the open book is a torus; the number of boundary components is $N=b_{1}+b_{2}+ \cdots +b_{r-1}- 2r +3$ (cf. Figure \ref{dihedral1}). The monodromy is
\[
T_1 T_2 \cdots T_{r-2} W_{r-1} ( t_{c_{1}}\cdots t_{c_{r-2}} ) (t_{\alpha_{1}}t_{\alpha_{2}} t_{\beta} )^{2}.
\]
\begin{figure}[hbt]
 \begin{center}
    \includegraphics[width=12cm]{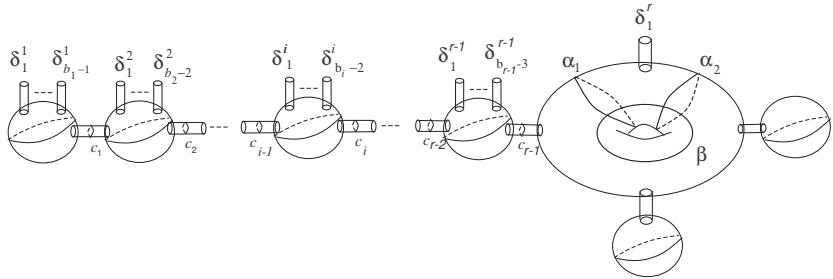}
  \caption{The page for dihedral singularity for $r>2$, $b = 2$ and $b_{r-1} > 2$.}
  \label{dihedral1}
   \end{center}
 \end{figure}

Suppose now that $k<r-1$. Taking $m_i=1$ for $1\leq i\leq k$, $m_j=2$ for $k+1 \leq j\leq r$, $m_{r+1}=m_{r+2}=1$,
so that
\[
\underline{m} =(1,1,\ldots,1,2,\ldots,2,2,1,1),
\]
we find
\[
\underline{n}= (b_{1}-1,b_{2}-2,b_{3}-2, \ldots , b_{k-1}-2, b_{k}-3,1,0,\ldots, 0).
\]
It follows again from the construction in Section~\ref{costruction} that the page of the open book
is a torus with $N=b_{1}+b_{2}+ \cdots +b_{k}- 2k+1$ boundary components (cf. Figure \ref{dihedral2}).
For the monodromy, we glue the following maps:
\begin{enumerate}
  \item[] $\phi|_{U^{1}_{j}}= t_{\delta^{1}_{j}}$, $j=1,\ldots,b_{1}-1$,
  \item[] $\phi|_{U^{i}_{j}}= t_{\delta^{i}_{j}}$, $i=2,\ldots,k-1$ and $j=1, \ldots ,b_{i}-2$,
  \item[] $\phi|_{U^{k}_{j}}= t_{\delta^{k}_{j}}$ for $j=1, \ldots , b_{k}-3$,
  \item[] $\phi|_{U^{i,i+1}_{1}}= t_{c_{i}}$, $i=1,\ldots,k-1$,
  \item[] $\left(\phi|_{U^{k,k+1}_{1}}\right)^2 = t_{c_{k}}$,
  \item[] $\left(\phi|_{U^{k+1}_{1}}\right)^2 = t_{\delta^{k+1}_{1}}$,
  \item[] $\left(\phi|_{U^{i,i+1}_{j}}\right)^2= t_{c_{i}}$, $i=k+1,\ldots,r-1$ and $j=1,2$.
\end{enumerate}
In this case, for $i=k+1,\ldots,r-1$ the curves $c_{i}$ are isotopic to $\alpha_{2}$. By the two--holed torus relation (\ref{2holed1}) for the torus with boundary $c_{k}$ and $\delta ^{k+1}_{1}$ and by Theorem~\ref{mcg}, we get the monodromy
to be
\[
T_1 T_2 \cdots  T_{k-1} W_{k} ( t_{c_{1}}\cdots t_{c_{k-1}} ) (t_{\alpha_{1}}t_{\alpha_{2}} t_{\beta} )^{2} \left( t_{\alpha_{2}}\right) ^{r-(k+1)}.
\]
\begin{figure}[hbt]
 \begin{center}
    \includegraphics[width=12cm]{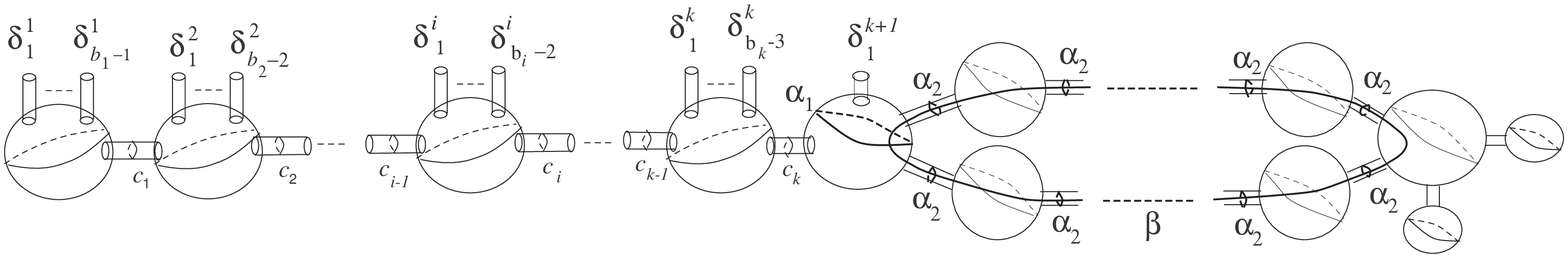}
  \caption{The page for dihedral singularity for $r>2$, $b = b_{r-1}=\cdots= b_{k+1}=2$, $b_{k} > 2$ and $1\leq  k <r-1$.}
  \label{dihedral2}
   \end{center}
 \end{figure}
\end{proof}

\subsection{Tetrahedral quotient singularities}
In this subsection, we investigate the tetrahedral quotient singularities. We write explicitly the Milnor open book decomposition supporting the unique Milnor fillable contact structure on the link of a tetrahedral quotient singularity.
\begin{proposition} \label{prop:tetrahedral} {\rm \textbf{(Tetrahedral quotient singularities)}}
The unique Milnor fillable contact structure on the link of a tetrahedral quotient
singularity is supported by a planar open book decomposition (resp.
a genus-$1$ open book decomposition) if $b > 2$ (resp. $b=2$). The number $N$ of
boundary components of the page and the monodromy $\phi$ are given as follows (cf.
Figure~\ref{tetrahedral_pl} and Figure~\ref{tetrahedral_gen}):
For $b > 2$
\[\displaystyle
(N,\phi) =\left\{
\begin{array}{ll}\displaystyle
    \left( b,\ \, (t_{\delta^{1}_{1}})^{3} ( t_{\delta^{3}_{1}}\cdots t_{\delta^{3}_{b-3}} ) (t_{\delta^{5}_{1}})^{3} (t_{\delta^{6}_{1}})^{2}\right) ,  & \mbox{ in the case } (i) \\
     \left( b+1,\ \, t_{\delta^{1}_{1}} t_{\delta^{1}_{2}} ( t_{\delta^{2}_{1}}\cdots t_{\delta^{2}_{b-3}} ) (t_{\delta^{4}_{1}})^{3} (t_{\delta^{5}_{1}})^{2} t_{c_{1}}\right), &  \mbox{ in the case } (ii) \\
    \left( b+2, \ \, t_{\delta^{1}_{1}} t_{\delta^{1}_{2}} ( t_{\delta^{2}_{1}}\cdots t_{\delta^{2}_{b-3}} ) t_{\delta^{3}_{1}} t_{\delta^{3}_{2}} (t_{\delta^{4}_{1}})^{2} t_{c_{1}}  t_{c_{2}}\right), &
     \mbox{ in the case } (iii)
\end{array}
\right.
\]
and for $b = 2$
\[\displaystyle
(N,\phi) =\left\{
\begin{array}{ll}\displaystyle
    \left( 1,\ \, (t_{\alpha} t_{\beta} )^{4} \right) , & \mbox{ in the case } (i) \\
     \left( 2,\ \, t_{\delta^1_{1}} (t_{\alpha_{1}}t_{\alpha_{2}} t_{\beta} t_{\alpha_{1}}t_{\alpha_{2}} t_{\beta} t_{\alpha_{2}}) \right), &
     \mbox{ in the case } (ii) \\
    \left( 3,\ \, t_{\delta^1_{1}} t_{\delta^{3}_1} (t_{\alpha_{1}}t_{\alpha_{3}} t_{\beta} t_{\alpha_{2}}t_{\alpha_{3}} t_{\beta} ) \right), &
     \mbox{ in the case } (iii).
\end{array}
\right.
\]
\end{proposition}

\begin{proof}
\textbf{Case $1$:} $b > 2$. We analyze each case.

$(i)$ Following the steps in the construction given in Section~\ref{costruction}, we choose $\underline{m} = (1,1,1,1,1,1)$. From the equation (\ref{first}) we find $\underline{n}= (1,0,b-3,0,1,1)$. Since $m_i=1$ for all $i$, the surface $F_i$ at the vertex $A_i$
is a sphere with $v_{i}+ n_{i}$ boundary components.
We connect these spheres $F_i$ and $F_j$ with an annulus if the vertices $A_i$ and $A_j$
are connected by an edge. It follows that the page of the open book is a sphere with
$n_1+\cdots + n_6= b$ boundary components (cf. Figure~\ref{tetrahedral_pl}$(i)$).
The monodromy $\phi$ restricted to each annuli is given as
\begin{enumerate}
  \item[$\bullet$] $\phi|_{U^{1}_{1}}= t_{\delta^{1}_{1}}$,
  \item[$\bullet$] $\phi|_{U^{3}_{j}}= t_{\delta^{3}_{j}}$, for $j=1, \ldots, b-3$,
  \item[$\bullet$] $\phi|_{U^{5}_{1}}= t_{\delta^{5}_{1}}$,
  \item[$\bullet$] $\phi|_{U^{6}_{1}}= t_{\delta^{6}_{1}}$,
  \item[$\bullet$] $\phi|_{U^{i,i+1}_{1}}= t_{c_{i}}$, for $i=1, \ldots, 4$, and
  \item[$\bullet$] $\phi|_{U^{3,6}_{1}}= t_{c_{5}}$.
\end{enumerate}

Note that the curves $ c_1$ and $c_2$ are isotopic to $\delta^{1}_{1}$, $ c_3$ and $c_4$ are isotopic to $\delta^{5}_{1}$, and $ c_5$ is isotopic to $\delta^{6}_{1}$. From this we find that
\[
\phi= (t_{\delta^{1}_{1}})^{3} ( t_{\delta^{3}_{1}}\cdots t_{\delta^{3}_{b-3}} ) (t_{\delta^{5}_{1}})^{3} (t_{\delta^{6}_{1}})^{2}.
\]

$(ii)$ Choosing $\underline{m} = (1,1,1,1,1)$ gives $\underline{n}= (2,b-3,0,1,1)$. Following the construction above we easily get the desired open book (cf. Figure~\ref{tetrahedral_pl}$(ii)$).

$(iii)$ Taking $\underline{m} = (1,1,1,1)$ gives $\underline{n}= (2,b-3,2,1)$. The rest of the proof is the same as the case $(i)$ (cf. Figure~\ref{tetrahedral_pl}$(iii)$).

\textbf{Case $2$:} $b = 2$. Again we investigate each of the three cases.

This type of singularity has a graph $\Gamma$ given in Figure~\ref{tetrahedral}. For each vertex $A_{i}$ of $\Gamma$, there is a sphere with $v_{i} + n_{i}$ boundary components, where $v_i$ is the valency of the vertex $A_i$ and $n_i$ is calculated from equation (\ref{first}).

$(i)$ Choosing $\underline{m} = (1,2,3,2,1,2)$ gives $\underline{n}= (0,0,0,0,0,1)$. To construct the page, we plumb the surfaces $F_i$, according to the graph $\Gamma$. The surface $F_i$ is the $m_i$--cover of the sphere with $v_i+n_i$ boundary components. The number of boundary components of $F_{i}$ is
\[
n_{i}+ \sum_{(A_{i},A_{j})\in \mathcal{E}} \gcd{(m_{i},m_{j})} = n_i + v_i.
\]
Hence the surface $F_1$ is a sphere with one boundary component,
$F_2$ is a sphere with two boundary components,
$F_3$ is a torus with three boundary components,
$F_4$ is a sphere with two boundary components,
$F_5$ is sphere with one boundary component, and
$F_6$ is sphere with two boundary components.

The page of the open book is constructed by connecting these surfaces with an annulus according to the graph $\Gamma$ (cf. Figure~\ref{tetrahedral_gen}$(i)$). Thus the page is a one--holed torus. The monodromy can easily be calculated by gluing the monodromy restricted to the annuli $U^{3,6}_{1}$ and $U^{6}_{1}$.
\[
\phi= (\phi|_{U^{3,6}_{1}} ) (\phi|_{U^{6}_{1}}).
\]
From Section~\ref{costruction}, we have
\begin{enumerate}
  \item[$\bullet$] $\left(  \phi|_{U^{3,6}_{1}} \right)^6 = t_{c_{5}}$, and
  \item[$\bullet$] $\left(  \phi|_{U^{6}_{1}} \right)^2= t_{\delta^{6}_{1}}$.
\end{enumerate}
Since the curves $c_5$ and $\delta^{6}_{1}$ are isotopic, we may write
\begin{enumerate}
  \item[] $ \phi ^3 = \left( (\phi|_{U^{3,6}_{1}} ) (\phi|_{U^{6}_{1}})  \right)^3 $ \\
  \item[] $ \ \ \ \ = \left( t_{\delta^{6}_{1}} \right)^{2} $  \\
  \item[] $ \ \ \ \ = \left( t_{\alpha} t_{\beta} \right)^{12}, $
\end{enumerate}
by using the once-punctured torus relation (\ref{1holed}).
By Theorem~\ref{mcg}, we find
\[
\phi= \left( t_{\alpha} t_{\beta} \right)^4.
\]

$(ii)$ We take $\underline{m} = (1,2,2,1,1)$ and get $\underline{n}= (1,0,1,0,0)$.
It follows that the page of the open book is a torus with two boundary components
(cf. Figure~\ref{tetrahedral_gen}$(ii)$). The monodromy can easily be calculated.

$(iii)$ Choosing $\underline{m} = (1,2,1,1)$ gives $\underline{n}= (1,1,1,0)$. Then
the page of the open book is a torus with three boundary components (cf. Figure~\ref{tetrahedral_gen}$(iii)$). The monodromy is found to be as stated.
\end{proof}

\begin{figure}[hbt]
 \begin{center}
    \includegraphics[width=9cm]{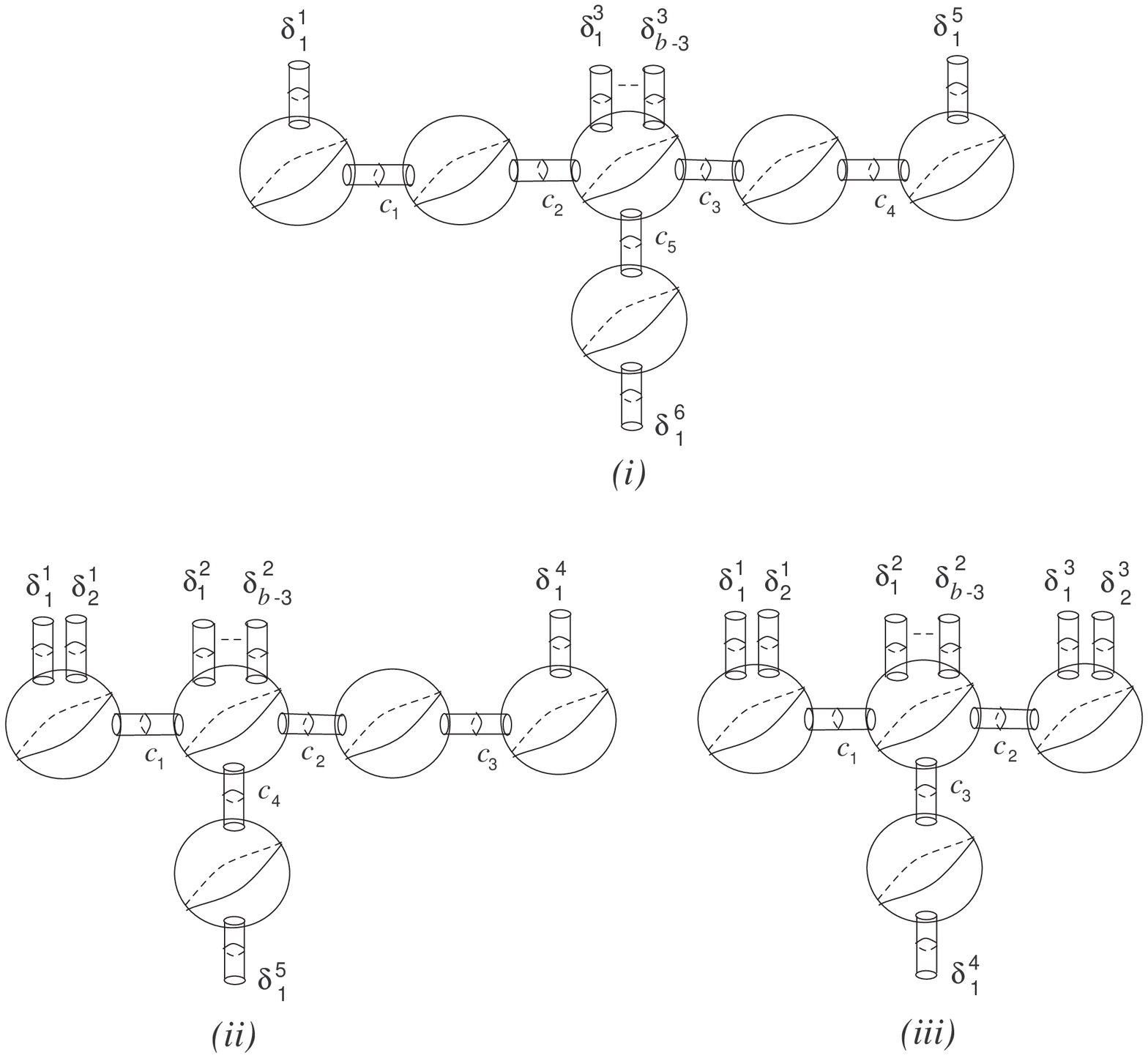}
  \caption{The page for tetrahedral singularities for $b>2$.}
  \label{tetrahedral_pl}
   \end{center}
 \end{figure}

\begin{figure}[hbt]
 \begin{center}
    \includegraphics[width=9cm]{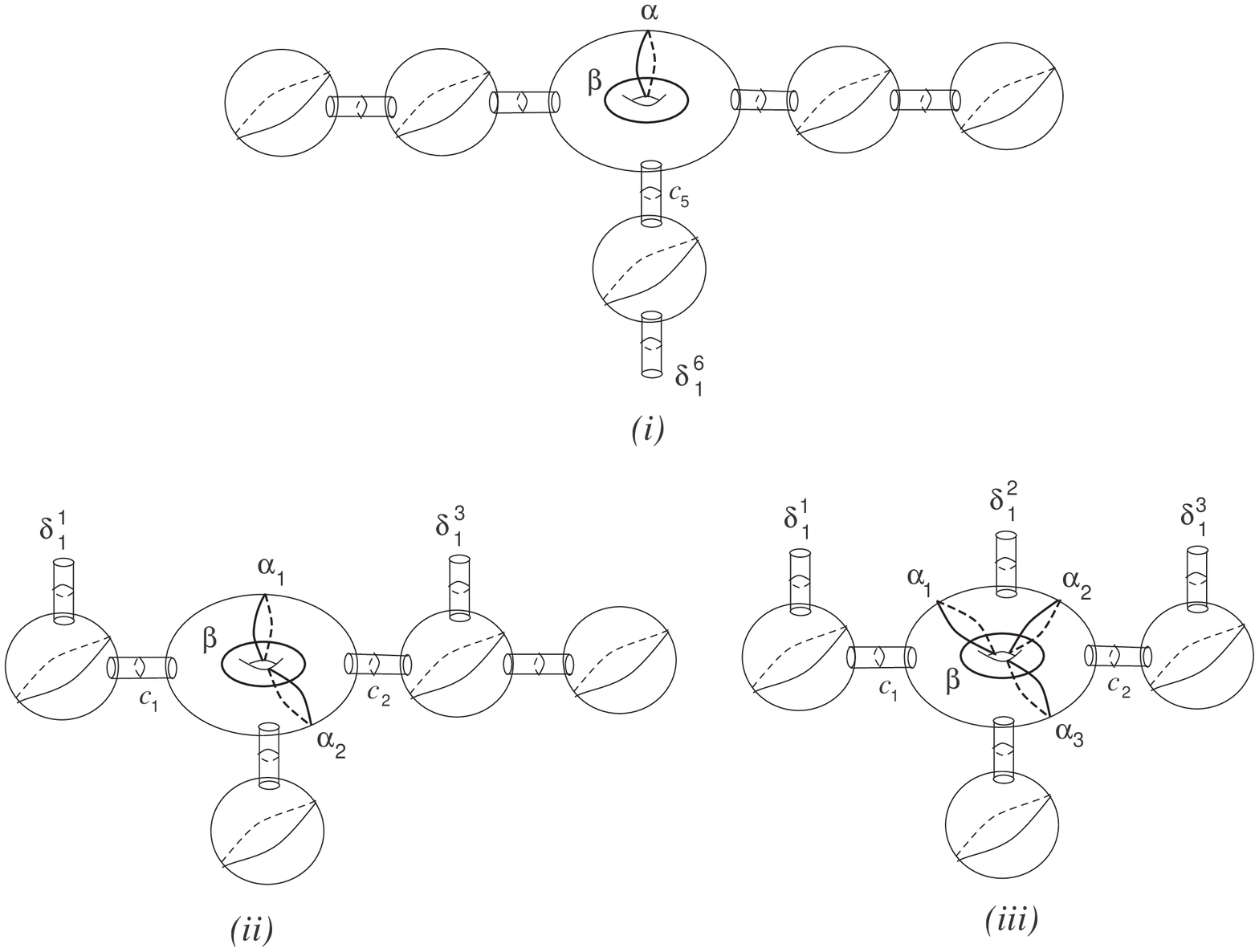}
  \caption{The page for tetrahedral singularities for $b=2$.}
  \label{tetrahedral_gen}
   \end{center}
 \end{figure}

\subsection{Octahedral quotient singularities}
In the proposition below, we construct the Milnor open book decomposition supporting the unique Milnor fillable contact structure
on the link of an octahedral quotient singularity.
\begin{proposition} \label{prop:octahedral} {\rm \textbf{(Octahedral quotient singularities)}}
The unique Milnor fillable contact structure on the link of an octahedral quotient
singularity is supported by a planar open book decomposition if $b > 2$ and
a genus-$1$ open book decomposition if $b=2$. The number $N$ of
boundary components of the page and the monodromy $\phi$ are given as follows (cf.
Figure~\ref{octahedral}):
For $b > 2$
\[\displaystyle
(N,\phi) =\left\{
\begin{array}{ll}\displaystyle
    \left( b,\ \, (t_{\delta^{1}_{1}})^{4} ( t_{\delta^{4}_{1}}\cdots t_{\delta^{4}_{b-3}} ) (t_{\delta^{6}_{1}})^{3} (t_{\delta^{7}_{1}})^{2} (t_{\delta^{6}_{1}})^{2}\right) ,  & \mbox{ in the case } (i) \\
     \left( b+1,\ \, (t_{\delta^{1}_{1}})^4 ( t_{\delta^{4}_{1}}\cdots t_{\delta^{4}_{b-3}} ) t_{\delta^{5}_{1}} t_{\delta^{5}_{2}} (t_{\delta^{6}_{1}})^2 t_{c_{1}}\right), &  \mbox{ in the case } (ii) \\
    \left( b+2, \ \, t_{\delta^{1}_{1}} t_{\delta^{1}_{2}} ( t_{\delta^{2}_{1}}\cdots t_{\delta^{2}_{b-3}} ) t_{\delta^{3}_{1}} t_{\delta^{3}_{2}} (t_{\delta^{4}_{1}})^{2} t_{c_{1}}  t_{c_{2}}\right), &
     \mbox{ in the case } (iii)\\
      \left( b+3, \ \, t_{\delta^{1}_{1}} t_{\delta^{1}_{2}} t_{\delta^{1}_{3}} ( t_{\delta^{2}_{1}}\cdots t_{\delta^{2}_{b-3}} ) (t_{\delta^{4}_{1}})^{3} (t_{\delta^{5}_{1}})^{2} t_{c_{1}} \right), &
     \mbox{ in the case } (iv)
\end{array}
\right.
\]
and for $b = 2$
\[\displaystyle
(N,\phi) =\left\{
\begin{array}{ll}\displaystyle
    \left( 1,\ \, t_{\beta} (t_{\alpha} t_{\beta} )^{4} \right) , & \mbox{ in the case } (i) \\
     \left( 2,\ \, t_{\delta_{2}} (t_{\alpha_{1}}t_{\alpha_{2}} ( t_{\alpha_{1}}t_{\alpha_{2}} t_{\beta})^2 \right), &
     \mbox{ in the case } (ii) \\
    \left( 3,\ \, t_{\delta_{1}} t_{\delta_{2}} (t_{\alpha_{1}}t_{\alpha_{2}} t_{\beta} t_{\alpha_{1}}t_{\alpha_{2}} t_{\beta} t_{\alpha_{2}} ) \right), &
     \mbox{ in the case } (iii)\\
      \left( 4,\ \, t_{\delta_{1}} t_{\delta_2} t_{\delta_{4}} (t_{\alpha_{1}}t_{\alpha_{3}} t_{\beta} t_{\alpha_{2}}t_{\alpha_{3}} t_{\beta} ) \right), &
     \mbox{ in the case } (iv).
\end{array}
\right.
\]
\end{proposition}

\begin{proof}
\textbf{Case $1$:} $b > 2$.

$(i)$ We choose $\underline{m} = (1,1,1,1,1,1,1)$ from which we find $\underline{n}= (1,0,0,b-3,0,1,1)$. Since $m_i=1$ for all $i$, the surface $F_i$ at the vertex $A_i$
is a sphere with $v_{i}+ n_{i}$ boundary components.
We connect these spheres $F_i$ and $F_j$ with an annulus if the vertices $A_i$ and $A_j$
are connected by an edge. It follows that the page of the open book is a sphere with
$n_1+\cdots + n_7= b$ boundary components (cf. Figure~\ref{octahedral_pl}$(i)$).
The monodromy $\phi$ restricted to each annuli is given as
\begin{enumerate}
  \item[$\bullet$] $\phi|_{U^{1}_{1}}= t_{\delta^{1}_{1}}$,
  \item[$\bullet$] $\phi|_{U^{4}_{j}}= t_{\delta^{4}_{j}}$, for $j=1, \ldots, b-3$,
  \item[$\bullet$] $\phi|_{U^{6}_{1}}= t_{\delta^{6}_{1}}$,
  \item[$\bullet$] $\phi|_{U^{7}_{1}}= t_{\delta^{7}_{1}}$,
  \item[$\bullet$] $\phi|_{U^{i,i+1}_{1}}= t_{c_{i}}$, for $i=1, \ldots, 5$, and
  \item[$\bullet$] $\phi|_{U^{4,7}_{1}}= t_{c_{6}}$.
\end{enumerate}

The curves $ c_1$, $c_2$ and $c_3$ are isotopic to $\delta^{1}_{1}$, $ c_3$ and $c_4$ are isotopic to $\delta^{6}_{1}$, and $ c_6$ is isotopic to $\delta^{6}_{1}$. From this we find that
\[
\phi= (t_{\delta^{1}_{1}})^{4} ( t_{\delta^{4}_{1}}\cdots t_{\delta^{4}_{b-3}} ) (t_{\delta^{6}_{1}})^{3} (t_{\delta^{7}_{1}})^{2}.
\]

Similarly, in order to construct the Milnor open books in part $(ii), (iii)$ and $(iv)$, we choose $\underline{m}$ and $\underline{n}$ as in the table below. Following the construction steps, one can easily get the open book stated in Proposition~\ref{prop:octahedral}.

\[
\begin{tabular}{|c|c|c|}
\hline
Part & $\underline{m}$ & $\underline{n}$\\
\hline
$(ii)$  & $(1,1,1,1,1,1) $ & $(1,0,0,b-3,2,1) $ \\
 \hline
$(iii)$  & $(1,1,1,1,1) $ & $(3,b-3,0,1,1) $ \\
 \hline
$(iv)$  & $(1,1,1,1) $ & $(3,b-3,2,1) $ \\
 \hline

\end{tabular}
\]

\textbf{Case $2$:} $b = 2$. Again we investigate each of the four cases.

This type of singularity has a graph $\Gamma$ given in Figure~\ref{octahedral}. For each vertex $A_{i}$ of $\Gamma$, there is a sphere with $v_{i} + n_{i}$ boundary components, where $v_i$ is the valency of the vertex $A_i$ and $n_i$ is calculated from equation (\ref{first}).

$(i)$ Choosing $\underline{m} = (1,2,3,4,3,2,2)$ gives $\underline{n}= (0,0,0,0,0,1,0)$. To construct the page, we plumb the surfaces $F_i$, according to the graph $\Gamma$. The surface $F_i$ is the $m_i$--cover of the sphere with $v_i+n_i$ boundary components. There are $  n_i + v_i $ boundary components of the surface $F_{i}$.
Hence, the surfaces $F_1$ and $F_7$ are spheres with one boundary component. $F_2$, $F_3$, $F_5$ and $F_6$ are spheres with two boundary components and
$F_3$ is a torus with three boundary components.
The page of the open book is constructed by connecting these surfaces with an annulus according to the graph $\Gamma$ (cf. Figure~\ref{octahedral_gen}$(i)$). Thus the page is a one--holed torus. The monodromy can easily be calculated by gluing the monodromy restricted to the annuli $U^{4,5}_{1}$, $U^{5,6}_{1}$  and $U^{6}_{1}$.
\[
\phi= (\phi|_{U^{4,5}_{1}} ) (\phi|_{U^{5,6}_{1}} ) (\phi|_{U^{6}_{1}}).
\]
From the calculations of Section~\ref{costruction},
\begin{enumerate}
  \item[$\bullet$] $\left(  \phi|_{U^{4,5}_{1}} \right)^{12} = t_{c_{4}}$,
  \item[$\bullet$] $\left(  \phi|_{U^{3,6}_{1}} \right)^6 = t_{c_{5}}$, and
  \item[$\bullet$] $\left(  \phi|_{U^{6}_{1}} \right)^2= t_{\delta^{6}_{1}}$.
\end{enumerate}
Since the curves $c_4$, $c_5$ and $\delta^{6}_{1}$ are isotopic, we may write
\begin{enumerate}
  \item[] $ \phi ^4 = \left( (\phi|_{U^{4,5}_{1}} ) (\phi|_{U^{5,6}_{1}} ) (\phi|_{U^{6}_{1}})  \right)^4 $ \\
  \item[] $ \ \ \ \ = \left( t_{\delta^{6}_{1}} \right)^{3} $.  \\
\end{enumerate}
By using the once-punctured torus relation (\ref{1holed}),
\begin{enumerate}
  \item[] $ \phi ^4 = \left( (t_{\alpha} t_{\beta}) ^6 \right)^{3}. $
\end{enumerate}
Using the braid relations we may write
\begin{enumerate}
  \item[] $ \phi ^4 = \left( t_{\beta} (t_{\alpha} t_{\beta}))^4 \right)^{4}. $
\end{enumerate}
Using Theorem~\ref{mcg}, we obtain the monodromy as
\begin{enumerate}
  \item[] $ \phi  =  t_{\beta} (t_{\alpha} t_{\beta})^4.  $
\end{enumerate}

In order to prove the rest of the proposition, we take $\underline{m}$ and $\underline{n}$ as in the table below. Then following the construction steps, braid relations and Theorem~\ref{mcg} one can get Milnor open book given in Proposition~\ref{prop:octahedral}.

\[
\begin{tabular}{|c|c|c|}
\hline
Part & $\underline{m}$ & $\underline{n}$\\
\hline
$(ii)$  & $(1,2,2,2,1,1)$ & $ (0,1,0,0,1,0)$ \\
 \hline
$(iii)$  & $ (1,2,2,1,1)$ & $ (2,0,1,0,0)$ \\
 \hline
$(iv)$  & $(1,2,1,1) $ & $(2,1,1,0) $ \\
 \hline

\end{tabular}
\]
\end{proof}

\begin{figure}[hbt]
 \begin{center}
    \includegraphics[width=12cm]{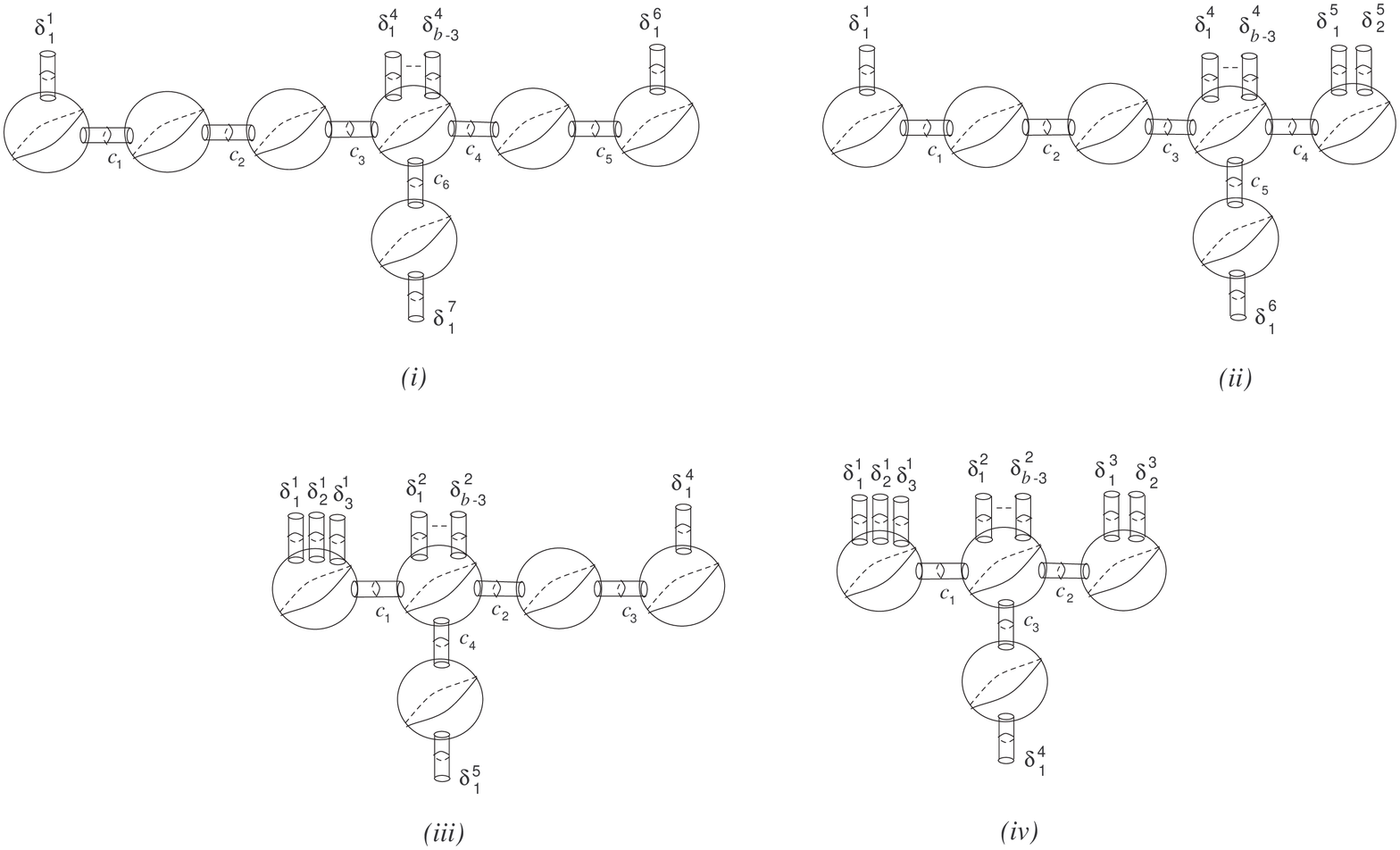}
  \caption{The page for octahedral singularities for $b>2$.}
  \label{octahedral_pl}
   \end{center}
 \end{figure}

\begin{figure}[hbt]
 \begin{center}
    \includegraphics[width=12cm]{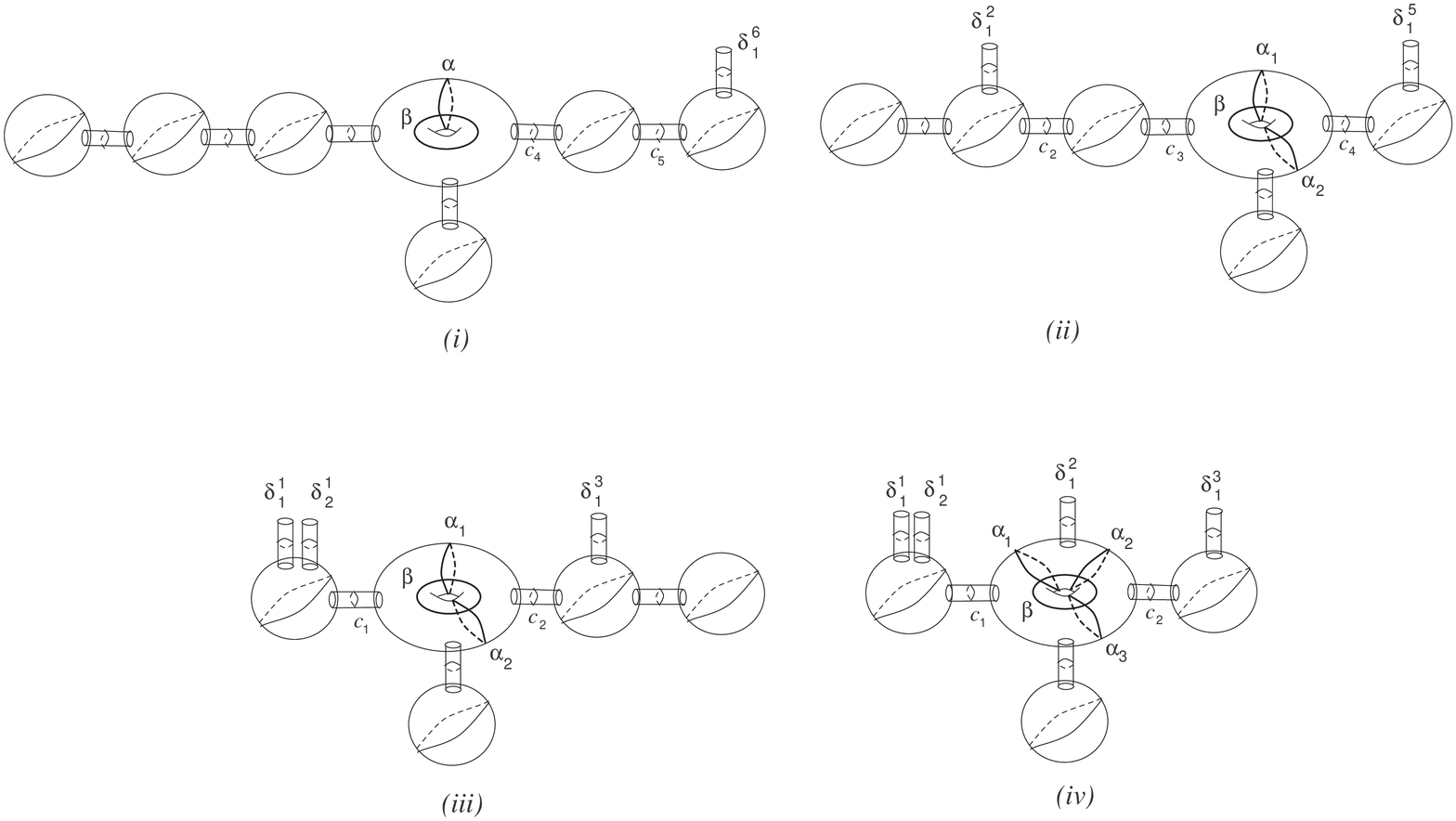}
  \caption{The page for octahedral singularities for $b=2$.}
  \label{octahedral_gen}
   \end{center}
 \end{figure}

\subsection{Icosahedral quotient singularities}
In this part, we write the Milnor open book decomposition supporting the unique Milnor fillable contact structure on the link of an icosahedral quotient singularity. Notice that for $b=2$, $(i)$ is the $E_8$ singularity, and its open book decomposition is constructed in~\cite{b}, and~\cite{eo}. We simplify the construction of~\cite{b}, and write it for the completeness of the paper.

\begin{proposition} \label{prop:icosahedral} {\rm \textbf{(Icosahedral quotient singularities)}}
The unique Milnor fillable contact structure on the link of an icosahedral quotient
singularity is supported by a planar open book decomposition if $b > 2$  and
a genus-$1$ open book decomposition if $b=2$. The number $N$ of
boundary components of the page and the monodromy $\phi$ are given as follows (cf.
Figure~\ref{icosahedral1}):
For $b > 2$
\[\displaystyle
(N,\phi) =\left\{
\begin{array}{ll}\displaystyle
    \left( b,\ \, (t_{\delta^{1}_{1}})^{5} ( t_{\delta^{5}_{1}}\cdots t_{\delta^{5}_{b-3}} ) (t_{\delta^{7}_{1}})^{3} (t_{\delta^{8}_{1}})^{2} \right) ,  & \mbox{ in the case } (i) \\
     \left( b+1,\ \, (t_{\delta^{1}_{1}} t_{\delta^{1}_{2}} ( t_{\delta^{3}_{1}}\cdots t_{\delta^{3}_{b-3}} ) (t_{\delta^{5}_{1}})^3 (t_{\delta^{6}_{1}})^2 (t_{c_{1}})^2\right), &  \mbox{ in the case } (ii) \\
    \left( b+2, \ \, (t_{\delta^{1}_{1}})^5 ( t_{\delta^{5}_{1}}\cdots t_{\delta^{5}_{b-3}} ) t_{\delta^{6}_{1}} t_{\delta^{6}_{2}} (t_{\delta^{7}_{1}})^{2} t_{c_{1}} \right), &
     \mbox{ in the case } (iii)\\
      \left( b+1, \ \, (t_{\delta^{1}_{1}})^2 t_{\delta^{2}_{1}} ( t_{\delta^{3}_{1}}\cdots t_{\delta^{3}_{b-3}} ) (t_{\delta^{5}_{1}})^{3} (t_{\delta^{6}_{1}})^{2} t_{c_{1}} \right), &
     \mbox{ in the case } (iv) \\
         \left( b+2,\ \, t_{\delta^{1}_{1}} t_{\delta^{1}_{2}} ( t_{\delta^{3}_{1}}\cdots t_{\delta^{3}_{b-3}}) t_{\delta^{4}_{1}} t_{\delta^{4}_{2}} (t_{\delta^{5}_{1}})^{2} (t_{c_{1}})^2 t_{c_{2}} \right) ,  & \mbox{ in the case } (v) \\
     \left( b+3,\ \, t_{\delta^{1}_{1}} t_{\delta^{1}_{2}} t_{\delta^{1}_{3}} t_{\delta^{1}_{4}} ( t_{\delta^{2}_{1}}\cdots t_{\delta^{2}_{b-3}} ) (t_{\delta^{4}_{1}})^3 (t_{\delta^{5}_{1}})^2  t_{c_{1}}\right), &  \mbox{ in the case } (vi) \\
    \left( b+2, \ \, (t_{\delta^{1}_{1}})^2 t_{\delta^{2}_{1}} ( t_{\delta^{3}_{1}}\cdots t_{\delta^{3}_{b-3}} ) t_{\delta^{4}_{1}} t_{\delta^{4}_{2}} (t_{\delta^{5}_{1}})^{2} t_{c_{1}}  t_{c_{2}}\right), &
     \mbox{ in the case } (vii)\\
      \left( b+4, \ \, t_{\delta^{1}_{1}} t_{\delta^{1}_{2}} t_{\delta^{1}_{3}} t_{\delta^{1}_{4}} ( t_{\delta^{2}_{1}}\cdots t_{\delta^{2}_{b-3}} ) t_{\delta^{3}_{1}}  t_{\delta^{3}_{1}} (t_{\delta^{4}_{1}})^{2} t_{c_{1}} t_{c_{2}}  \right), &
     \mbox{ in the case } (viii)
\end{array}
\right.
\]
and for $b = 2$
\[\displaystyle
(N,\phi) =\left\{
\begin{array}{ll}\displaystyle
    \left( 1,\ \,  (t_{\alpha} t_{\beta} )^{5} \right) , & \mbox{ in the case } (i) \\
     \left( 2,\ \, t_{\delta_{1}} (t_{\alpha_{1}} (t_{\alpha_{2}})^2 t_{\beta})^2 \right), &
     \mbox{ in the case } (ii) \\
    \left( 2,\ \, t_{\delta_{2}} t_{\alpha_{1}}t_{\alpha_{2}} (t_{\alpha_{1}}t_{\alpha_{2}} t_{\beta} t_{\alpha_{1}}t_{\alpha_{2}} t_{\beta} t_{\alpha_{2}} ) \right), &
     \mbox{ in the case } (iii)\\
      \left( 2,\ \, (t_{\delta_{1}})^2 (t_{\alpha_{1}}t_{\alpha_{2}} t_{\beta} t_{\alpha_{1}}t_{\alpha_{2}} t_{\beta} t_{\alpha_{2}}) \right), &
     \mbox{ in the case } (iv)\\
     \left( 4,\ \, t_{\delta_{1}} t_{\delta_{3}} (t_{\alpha_{1}}t_{\alpha_{2}} t_{\alpha_{3}} t_{\beta} t_{\alpha_{2}}t_{\alpha_{3}} t_{\beta} ) \right) , & \mbox{ in the case } (v) \\
     \left( 4,\ \,  t_{\delta_{1}} t_{\delta_{2}} (t_{\alpha_{1}}t_{\alpha_{2}} t_{\beta} t_{\alpha_{1}}t_{\alpha_{2}} t_{\beta} t_{\alpha_{2}}) \right), &
     \mbox{ in the case } (vi) \\
    \left( 3,\ \, (t_{\delta_{1}})^2 t_{\delta_{3}} (t_{\alpha_{1}}t_{\alpha_{3}} t_{\beta} t_{\alpha_{2}}t_{\alpha_{3}} t_{\beta} ) \right), &
     \mbox{ in the case } (vii)\\
      \left( 5,\ \, t_{\delta_{1}} t_{\delta_2} t_{\delta_{3}} t_{\delta_{5}} (t_{\alpha_{1}}t_{\alpha_{3}} t_{\beta} t_{\alpha_{2}}t_{\alpha_{3}} t_{\beta} ) \right), &
     \mbox{ in the case } (viii).
\end{array}
\right.
\]
\end{proposition}

\begin{proof}

\textbf{Case $1$:} $b > 2$.

$(i)$ Following the steps in the construction of the open book, choosing $\underline{m} = (1,1,1,1,1,1,1,1) $, we get $\underline{n}= (1,0,0,0,b-3,0,1,1)$. Since $m_i=1$ for all $i$, the surface $F_i$ at the vertex $A_i$
is a sphere with $v_{i}+ n_{i}$ boundary components.
We connect these spheres $F_i$ and $F_j$ with an annulus if the vertices $A_i$ and $A_j$
are connected by an edge. It follows that the page of the open book is a sphere with
$b$ boundary components (cf. Figure~\ref{icosahedral_pl}$(i)$).
The monodromy of the open book is calculated by gluing the maps below:
\[
\phi= (\phi|_{U^{1}_{1}} ) \left( \phi|_{U^{1,2}_{1}} \cdots \phi|_{U^{6,7}_{1}} \right)  (\phi|_{U^{7}_{1}} ) (\phi|_{U^{8}_{1}} ) (\phi|_{U^{5,8}_{1}}) \left(\phi|_{U^{5}_{1}} \cdots \phi|_{U^{5}_{b-3}} \right).
\]

The monodromy $\phi$ restricted to each of these annuli is given by
\begin{enumerate}
  \item[$\bullet$] $\phi|_{U^{1}_{1}}= t_{\delta^{1}_{1}}$,
  \item[$\bullet$] $\phi|_{U^{5}_{j}}= t_{\delta^{5}_{j}}$, for $j=1, \ldots, b-3$,
  \item[$\bullet$] $\phi|_{U^{7}_{1}}= t_{\delta^{7}_{1}}$,
  \item[$\bullet$] $\phi|_{U^{8}_{1}}= t_{\delta^{8}_{1}}$,
  \item[$\bullet$] $\phi|_{U^{i,i+1}_{1}}= t_{c_{i}}$, for $i=1, \ldots, 6$, and
  \item[$\bullet$] $\phi|_{U^{5,8}_{1}}= t_{c_{7}}$.
\end{enumerate}

Note that the curves $ c_1$, $c_2$, $c_3$, $c_4$ are isotopic to $\delta^{1}_{1}$; $ c_5$, $c_6$ are isotopic to $\delta^{7}_{1}$; and $ c_7$ is isotopic to $\delta^{8}_{1}$. Hence the monodromy is

\[
\phi= (t_{\delta^{1}_{1}})^{5}  (t_{\delta^{7}_{1}})^{3} (t_{\delta^{8}_{1}})^{2} ( t_{\delta^{5}_{1}}\cdots t_{\delta^{5}_{b-3}} ) .
\]

In order to prove the other parts, we follow the construction steps given in Section~\ref{costruction}. Taking $\underline{m}$ and $\underline{n}$ as in the table below, one can get the desired open book.

\[
\begin{tabular}{|c|c|c|}
\hline
Part & $\underline{m}$ & $\underline{n}$\\
\hline
$(ii)$  & $(1,1,1,1,1,1)$ & $(2,0,b-3,0,1,1)$ \\
 \hline
$(iii)$  & $(1,1,1,1,1,1,1)$ & $(1,0,0,0,b-3,2,1)$ \\
 \hline
$(iv)$  & $(1,1,1,1,1,1)$& $(1,1,b-3,0,1,1)$ \\
 \hline
$(v)$  & $(1,1,1,1,1)$ & $(2,0,b-3,2,1)$ \\
 \hline
$(vi)$  & $(1,1,1,1,1)$ & $(4,b-3,0,1,1)$ \\
 \hline
$(vii)$  & $(1,1,1,1,1)$ & $(1,1,b-3,2,1)$ \\
 \hline
$(viii)$  & $(1,1,1,1)$ & $(4,b-3,2,1)$ \\
\hline

\end{tabular}
\]

\textbf{Case $2$:} $b = 2$.

$(i)$ The intersection matrix $I(\Gamma)$ of this singularity is
 \[
\begin{bmatrix}
 -2 & 1 & 0 & 0 &  0 & 0 &  0 & 0\\
  1 & -2& 1 & 0 &  0 & 0 &  0 & 0\\
  0 & 1 &-2 & 1 &  0 & 0 &  0 & 0\\
  0 & 0 & 1 &-2 &  1 & 0 &  0 & 0\\
  0 & 0 & 0 & 1 & -2 & 1 &  0 & 1\\
  0 & 0 & 0 & 0 &  1 &-2 &  1 & 0\\
  0 & 0 & 0 & 0 &  0 & 1 & -2 & 0\\
  0 & 0 & 0 & 0 &  1 & 0 &  0 & -2\\
\end{bmatrix}.
\]
In this case, we take $\underline{m}=(2,3,4,5,6,4,2,3)$.
From equation (\ref{first}), we find that $\underline{n}=(1,0,0,0,0,0,0,0)$. The page $\Sigma$ of the open book associated to $\underline{m}$ is a torus with two boundary components, built up as the union of five spheres with two boundary components, two spheres with one boundary component, a torus with three boundary components and eight annuli (see Figure \ref{icosahedral_gen}$(i)$). The monodromy restricted to each annulus is computed to be
\begin{enumerate}
  \item[$\bullet$] $(\phi|_{U^{1}_{1}})^2= t_{\delta^{1}_{1}}$,
  \item[$\bullet$] $(\phi|_{U^{1,2}_{1}}) ^{6}= t_{c_{1}}$,
  \item[$\bullet$] $(\phi|_{U^{2,3}_{1}}) ^{12}= t_{c_{2}}$,
  \item[$\bullet$] $(\phi|_{U^{3,4}_{1}}) ^{20}= t_{c_{3}}$,
  \item[$\bullet$] $(\phi|_{U^{4,5}_{1}}) ^{30}= t_{c_{4}}$.
\end{enumerate}
We get the monodromy $\phi$ of the open book by gluing these maps:
\[
\phi= ( \phi|_{U^{1}_{1}}) (\phi|_{U^{1,2}_{1}}) ( \phi|_{U^{2,3}_{1}}) ( \phi|_{U^{3,4}_{1}})(\phi|_{U^{4,5}_{1}}).
\]
The curves $c_{1}$, $c_{2}$, $c_{3}$ and $c_{4}$ are isotopic to $\delta^{1}_{1}$, and by using the above five equations, we may write
\[
( \phi )^6 = ( ( \phi|_{U^{1}_{1}}) (\phi|_{U^{1,2}_{1}}) ( \phi|_{U^{2,3}_{1}}) ( \phi|_{U^{6}_{1}}) ( \phi|_{U^{3,6}_{1}}))^6 =\left( t_{\delta^{1}_{1}}\right) ^5.
\]
Using the one--holed torus relation (\ref{1holed}),
we obtain
\[
( \phi )^6 = \left( t_{\delta^{1}_{1}}\right) ^5 =\left( t_{\alpha} t_{\beta} \right)^{30}.
\]
It follows now from Theorem \ref{mcg} that the monodromy $\phi$ of the open book is
\[
\phi=\left( t_{\alpha} t_{\beta} \right)^{5}.
\]

The rest of the proof is same as the part $(i)$. Choosing $\underline{m} $ and $\underline{n}$ as in table below, gives the open book decomposition stated in Proposition~\ref{prop:icosahedral}.

\[
\begin{tabular}{|c|c|c|}
\hline
Part & $\underline{m}$ & $\underline{n}$\\
\hline
$(ii)$  & $(1,2,3,2,1,2)$ & $(1,0,0,0,0,1)$ \\
 \hline
$(iii)$  & $(1,2,2,2,2,1,1)$ & $(0,1,0,0,0,1,0)$ \\
 \hline
$(iv)$  & $(1,1,2,2,1,1)$ & $(1,0,0,1,0,0)$ \\
 \hline
$(v)$  & $(1,2,2,1,1)$ & $(1,1,0,1,0)$ \\
 \hline
$(vi)$  & $(1,2,2,1,1)$ & $(3,0,1,0,0)$ \\
 \hline
$(vii)$  & $(1,1,2,1,1)$ & $(1,0,1,1,0)$ \\
 \hline
$(viii)$  & $(1,2,1,1)$ & $(3,1,1,0)$ \\
\hline

\end{tabular}
\]
\end{proof}

\begin{figure}[hbt]
 \begin{center}
    \includegraphics[width=12cm]{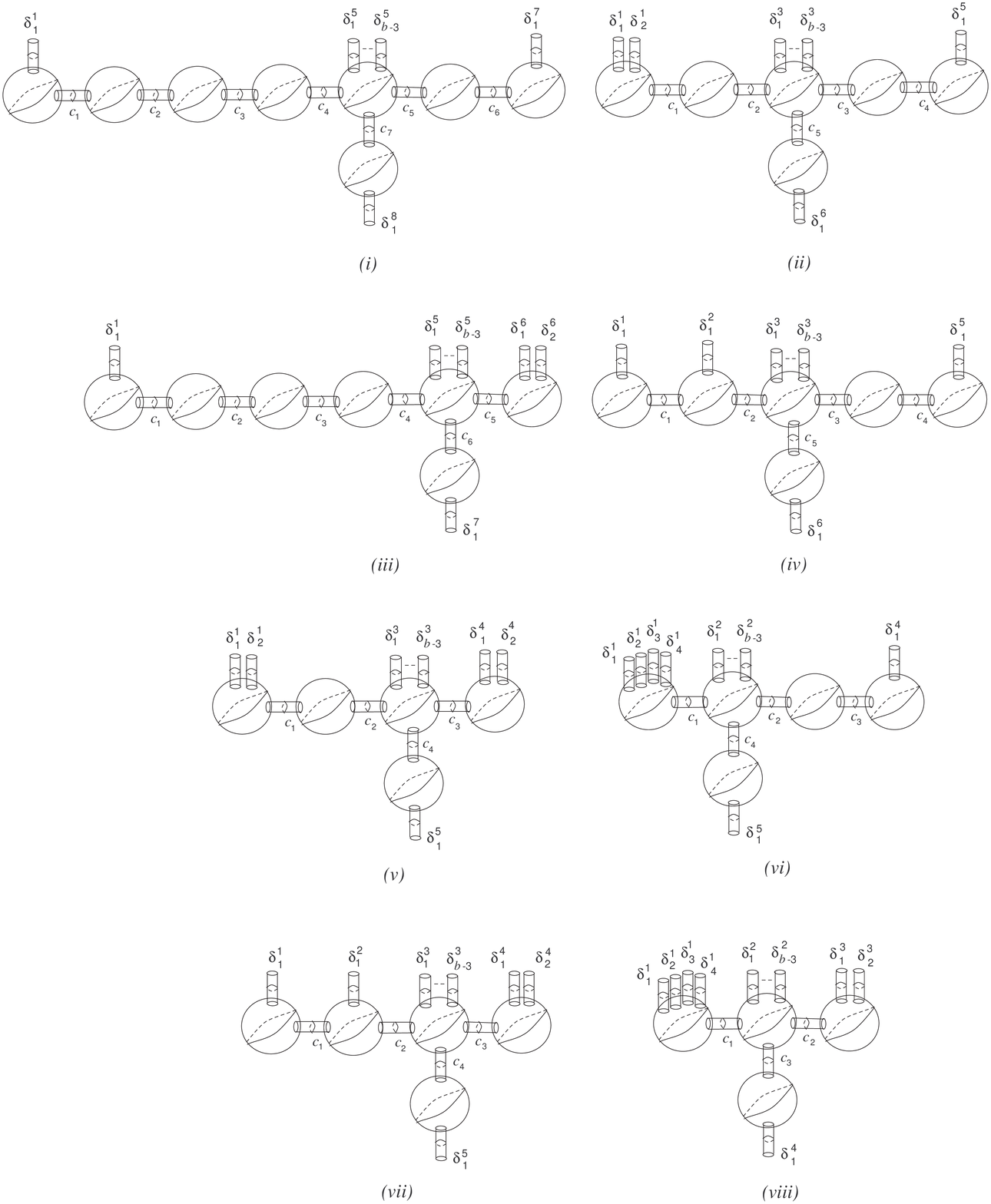}
  \caption{The page for icosahedral singularities for $b>2$.}
  \label{icosahedral_pl}
   \end{center}
 \end{figure}

\begin{figure}[hbt]
 \begin{center}
    \includegraphics[width=12cm]{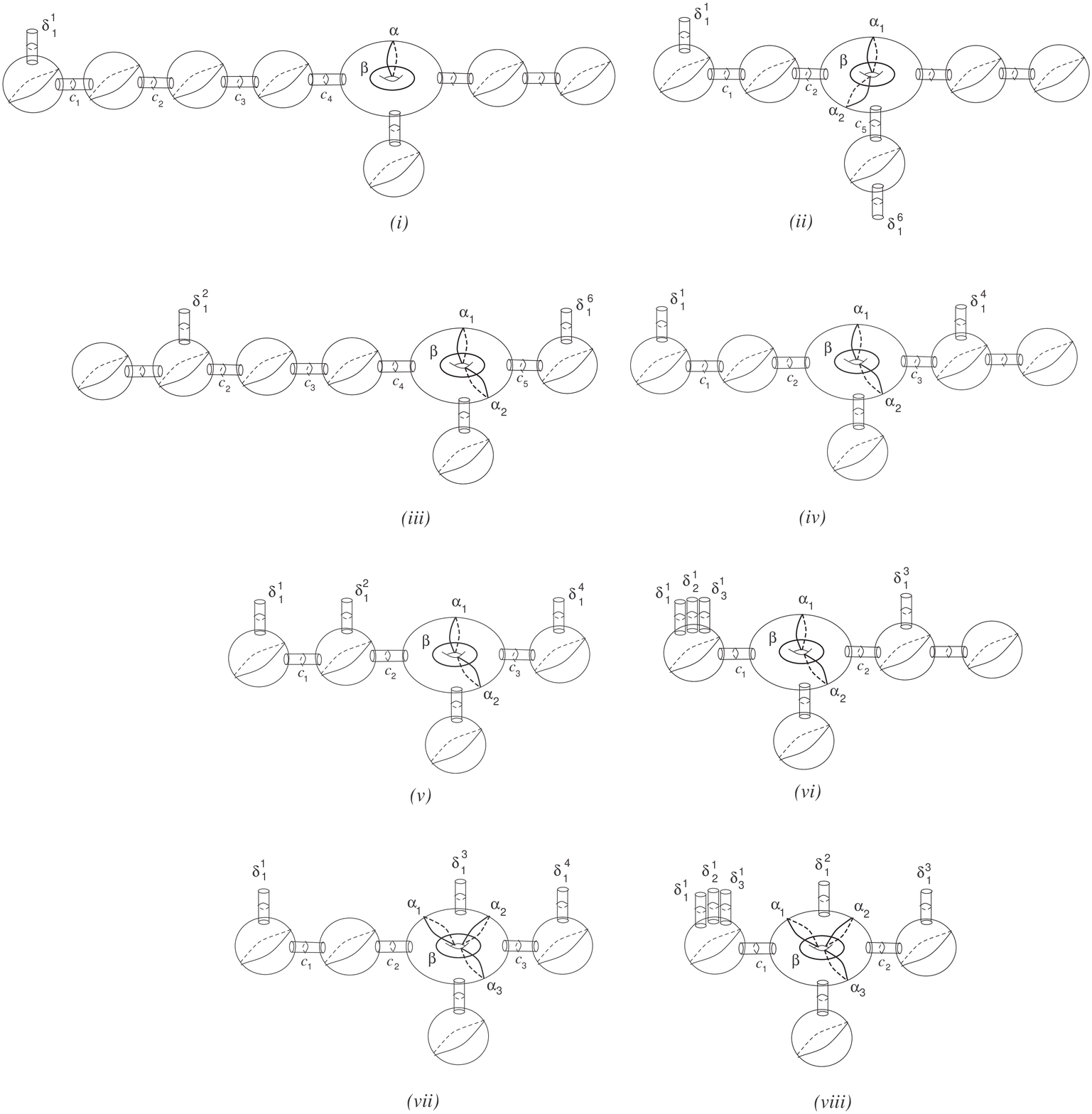}
  \caption{The page for icosahedral singularities for $b=2$.}
  \label{icosahedral_gen}
   \end{center}
 \end{figure}

\section{Proof of theorem \ref{mainth}}

In Section~\ref{lemmas} we have found the Milnor open book decompositions on the links of quotient surface singularities supporting the natural contact structure. Hence, we are able to say the following: The Milnor fillable contact structure on the link of a cyclic quotient surface singularity is supported by a planar open book (cf. Proposition~\ref{prop:cyclic}). Similarly, the natural contact structure on the links of other singularities in the case $b>2$ are all supported by planar open books (cf. Propositions~\ref{prop:dihedral}-\ref{prop:tetrahedral}-\ref{prop:octahedral}-\ref{prop:icosahedral}). Therefore, the support genus is the same as Milnor genus, which is zero, for these types.

In order to prove Theorem \ref{mainth}, we show that the unique Milnor fillable contact structure on the link of the quotient surface singularity cannot be supported by a planar open book for the following singularities: tetrahedral part $(i)$ for $b=2$; octahedral part $(i)$ for $b=2$, icosahedral part $(i)$ and $(ii)$ for $b=2$. These singularities have Milnor genus-$1$ open book decompositions, as shown in Section~\ref{lemmas}. Therefore, the Milnor genus is equal to the support genus for these types.

If X is a symplectic filling of a contact $3$-manifold $(M, \xi)$ and $\xi$ is supported by a
planar open book, then X can be embedded in $\#_{n} \overline{CP}^2$, connected sum of $n$ copies of $\overline{CP}^2$ by (the proof of) Theorem $1.2$ of \cite{e}. Hence, in order to show that the support genus of a symplectically fillable contact structure is positive, it suffices to show that
their symplectic fillings cannot be embedded in $\#_{n} \overline{CP}^2$.

  \begin{figure}[hbt]
 \begin{center}
    \includegraphics[width=4cm]{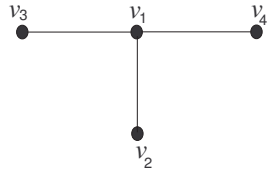}
  \caption{The intersection lattice $(\Z^{4},D_{4})$.}
  \label{fig1}
   \end{center}
 \end{figure}

Let $v_{1},v_{2},v_{3},v_{4}$ be the standard generators of the intersection lattice $(\Z^{4},D_{4})$ having
self-intersection $-2$, and $e_{1}, \ldots , e_{n}$ be the standard generators of $(\Z^{n},\D_{n}= \oplus _{n} \left\langle -1 \right\rangle)$ diagonal intersection lattice with self-intersection $-1$. By Lemma $3.1$ in \cite{l} (see also the proof of Theorem $4.2$ in \cite{ls}), there exists only one, up to composing with an
automorphism of $(\Z^{n},\D_{n})$, isometric embedding from $(\Z^{4},D_{4})$ to $(\Z^{n},\D_{n})$, which
sends $v_{1}$ to $e_{1} + e_{2}$, $v_{2}$ to $-e_{2} + e_{3}$, $v_{3}$ to $-e_{1} + e_{4}$ and $v_{4}$ to $-e_{2} -e_{3}$. The proof follows from the fact that, each $v_{i}$ has self-intersection $-2$, so that the image of $v_{i}$ under an isometric embedding must be of
the form $e_{j} + e_{k}$. From the intersection form of $D_{4}$, one can only get the above embedding (up to sign changes and permutations of generators of $(\Z^{n},\D_{n})$).

Let $L$ be any intersection lattice containing the sublattice with vertices $v_{1}, \ldots , v_{6}$ as
shown in Figure~\ref{fig2}, where $v_{1}, v_{2}, v_{3}, v_{4}$ have self-intersection $-2$. We prove for any
$n \geq 1$, there exists no isometric embedding from $L$ into $(\Z^{n},\D_{n})$.

  \begin{figure}[hbt]
 \begin{center}
    \includegraphics[width=8cm]{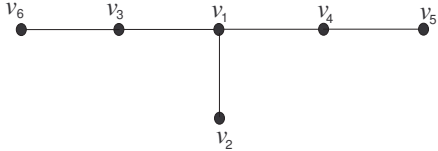}
  \caption{The sublattice.}
  \label{fig2}
   \end{center}
 \end{figure}

Suppose there exists such an isometric embedding $\varphi$. By the above discussion, we may assume that
\begin{itemize}
	\item $\varphi(v_{1})= e_{1} + e_{2} $,
	\item $\varphi(v_{2})= -e_{2} + e_{3} $,
	\item $\varphi(v_{3})= -e_{1} + e_{4} $, and
	\item$\varphi(v_{4})= -e_{2} - e_{3} $.
\end{itemize}
From the intersection form of $L$ one can see, that $v_{5}$ has an intersection with $v_{4}$. On the otherhand $v_{5}$
does not intersect $v_{2}$. Then one can get the equalities
below:
\begin{enumerate}
  \item[] $1 = \varphi(v_{5} \cdot v_{4}) = \varphi(v_{5}) \cdot \varphi(v_{4}) = \varphi(v_{5}) \cdot (-e_{2}- e_{3})$ and
  \item[] $0 = \varphi(v_{5} \cdot v_{2}) = \varphi(v_{5}) \cdot \varphi(v_{2}) = \varphi(v_{5}) \cdot (-e_{2}+e_{3})$.
\end{enumerate}
Hence we obtain
\begin{enumerate}
  \item[] $1 = \varphi(v_{5}) \cdot (-2e_{2})$,
\end{enumerate}
which is impossible.

Therefore, if one considers the intersection lattice $L$ as stated above and the natural
contact structure on the link of that plumbing, then its symplectic filling cannot be
embedded in $\#_{n} \overline{CP}^2$. So that contact structure cannot be supported by a planar open
book decomposition. Therefore the unique Milnor fillable contact structures on the
links of quotient surface singularities of tetrahedral part $(i)$ for $b=2$; octahedral part $(i)$ for $b=2$, icosahedral part $(i)$ and $(ii)$ for $b=2$ cannot be supported by planar
open book decompositions. These contact structures have support genus one.
For the remaining cases, we constructed minimal page-genus Milnor open books, and
the pages are all genus one surfaces. Hence, we conclude that support genus is at most
one for the corresponding contact structures. $\Box$

\begin{remark}
This method we used above, to prove the contact structures cannot be supported by planar open book decompositions, can be used to prove for some other symplectically fillable contact structures on different types of singularities/plumbings.

\end{remark}

\begin{remark}

Quotient surface singularities are rational surface singularities and the links of rational surface singularities are $L$-spaces. Hence we cannot use the obstructions in \cite{oss} for being supported by a planar open book decomposition.

\end{remark}

\begin{remark}
In this paper, we only investigate the Milnor open book decompositions supporting the canonical contact structure on the links of quotient surface singularities and the relation between the Milnor genus and the support genus. The relation between the binding number and support norm for this type of singularities could be understood with help of these Milnor open book decompositions.
\end{remark}

\end{document}